\documentclass[11pt,fleqn]{article}
\usepackage{amsmath,amsfonts,amssymb,calrsfs,amsthm}

\def\C{\mathbb C}

\def\R{\mathbb R}
\def\E{\mathbb E}

\def\Z{\mathbb Z}
\newcommand{\eps}{\varepsilon}
\def\bin#1#2{{#1\choose#2}}
\def\ds{\displaystyle}

\theoremstyle{plain}
\newtheorem{thm}{Theorem}[section]
\newtheorem{lemma}[thm]{Lemma}
\newtheorem{cor}[thm]{Corollary}
\newtheorem{example}[thm]{Example}
\newtheorem{definition}[thm]{Definition}
\newtheorem{proposition}[thm]{Proposition}
\theoremstyle{remark}
\newtheorem{remark}[thm]{Remark}

\author{Luciano Tubaro\\
}
\title{A simple method of computing Riemann-type integrals}
\date{}

\begin{document}
\sloppy
\maketitle

\begin{abstract}
The ``Riemann-type integrals'' of the title are not a single newly
defined integral, but the family of classical Riemann-sum constructions
--- Riemann, Riemann--Stieltjes, Young--Kondurar, complex-analytic, and
It\^o --- unified by the method below; we make this explicit at the
outset to avoid any suggestion of a single new integral being proposed.

In this article we show a very simple argument to compute directly the
Riemann integral of a polynomial and to establish directly the
``fundamental theorem of integral calculus'' for polynomials (and
subsequently for continuous functions); the argument may be of interest
from a didactic point of view. The same idea is applied to the real
line, to the complex plane, to the multidimensional case, to the
Riemann--Stieltjes integral, and finally to the It\^o integral of a
polynomial of a Wiener process (and, more generally, of a continuous
semimartingale), where one recovers the classical It\^o formula.

The contribution is methodological rather than a collection of
individually new results: a single algebraic telescoping identity, once
paired with an estimate on the quadratic remainder term
$\sum_i(\text{weight})_i\,(\Delta_i)^2$, drives every case treated here.
What changes from case to case is only the fate of that quadratic term:
it vanishes identically for polynomials on $\R$; it vanishes in the
limit whenever the underlying net has $Q(\pi)\to0$, on $\C$ and for
H\"older paths; and it neither vanishes nor is negligible for a
continuous semimartingale, where it survives as the quadratic variation
driving the It\^o correction term --- except for a planar Brownian
motion, or more generally a conformal martingale, where a purely
algebraic cancellation inside the complex square removes it again. We
regard this last point -- that the classical It\^o correction and its
complete cancellation in the conformal complex case are two instances of
the very same quadratic remainder, rather than two unrelated facts about
Brownian motion -- as the most distinctive single observation in the
paper, and return to it explicitly in the concluding remarks of
Part~III. The manuscript is organized in three parts accordingly:
Part~I develops the
core method on the real line and, via the same mechanism, on the complex
plane; Part~II extends it deterministically to the multidimensional,
Riemann--Stieltjes, and H\"older (Young--Kondurar) cases; Part~III
extends it to the stochastic setting, where the quadratic term's
non-vanishing is precisely what produces the It\^o correction.
\end{abstract}

\part{Core method}

\section{The real line}

Let $[a,b]$ be a finite interval of the real line and $f\colon[a,b]\to\R$
a bounded function. In order to define the \emph{Riemann sums} of $f$,
consider the family ${\cal P}(a,b)$ of all finite partitions
$a=t_0<t_1<\cdots<t_N=b$ of $[a,b]$, and in every subinterval
$[t_i,t_{i+1}]$ choose an arbitrary tag $\tau_i$; the corresponding
Riemann sum is
\[
\sum_{i=0}^{N-1}f(\tau_i)\,(t_{i+1}-t_i).
\]
On each tagged $\pi\in{\cal P}(a,b)$ we consider the structure of directed
set with respect to the mesh
\[
\delta(\pi)=\max_{0\le i\le N-1}(t_{i+1}-t_i);
\]
in this way the Riemann sums form a net on the directed set. The function
$f$ belongs to the class ${\cal R}$ of Riemann integrable functions if the
corresponding net of Riemann sums converges in $\R$ as $\delta(\pi)\to0$.
It is easy to verify that ${\cal R}$ is a vector space, containing all
the constants.

\subsection{The key result}\label{sec:real-key}

Given $\pi\in{\cal P}(a,b)$, say $a=t_0<t_1<\cdots<t_N=b$, and a fixed
integer $m$, introduce the $m+1$ numbers
\[
I_k(\pi)=\ds\sum_{i=0}^{N-1}t_i^{m-k}\,t_{i+1}^k(t_{i+1}-t_i),
\qquad k=0,\ldots,m.
\]

\begin{lemma}\label{lem:real1}
As $\delta(\pi)\to 0$, all the numbers $I_k(\pi)$ tend to the same limit
$L=\dfrac{b^{m+1}-a^{m+1}}{m+1}$.
\end{lemma}
\begin{proof}
Using the identity
\[
\lambda^{m+1}-\mu^{m+1}=(\lambda-\mu)(\lambda^m+\lambda^{m-1}\mu+\cdots+
\lambda\mu^{m-1}+\mu^m),
\]
it is easy to check that the sum of these quantities is the constant
\begin{equation}\label{real:e0}
I_0(\pi)+I_1(\pi)+\cdots+I_m(\pi)=b^{m+1}-a^{m+1}.
\end{equation}
For $1\le k\le m$ write $I_k(\pi)=I_0(\pi)+\rho_k(\pi)$, so that
\eqref{real:e0} becomes
\[
(m+1)I_0(\pi)+\sum_{k=1}^m\rho_k(\pi)=b^{m+1}-a^{m+1}.
\]
Now
\begin{align*}
\rho_k(\pi)&=\sum_{i=0}^{N-1}\big(t_i^{m-k}t_{i+1}^k-t_i^m\big)(t_{i+1}-t_i)\\
&=\sum_{i=0}^{N-1}t_i^{m-k}\big(t_{i+1}^{k-1}+t_{i+1}^{k-2}t_i+\cdots+t_i^{k-1}\big)(t_{i+1}-t_i)^2,
\end{align*}
and since every $t_i,t_{i+1}\in[a,b]$, there is a constant
$M=m(|a|+|b-a|)^{m-1}$, independent of $\pi$, such that
\[
|\rho_k(\pi)|\le M\sum_{i=0}^{N-1}(t_{i+1}-t_i)^2\le M(b-a)\,\delta(\pi).
\]
Hence, for $\delta(\pi)\to0$, all $\rho_k(\pi)\to0$ and consequently, from
\eqref{real:e0},
\[
\big|(m+1)I_0(\pi)-(b^{m+1}-a^{m+1})\big|
=\Big|\sum_{k=1}^m\rho_k(\pi)\Big|\le mM(b-a)\,\delta(\pi):
\]
every $I_k(\pi)$ converges to the same limit
\[
I_k(\pi)\longrightarrow L=\frac{b^{m+1}-a^{m+1}}{m+1}. \qedhere
\]
\end{proof}

\begin{remark}\label{rem:dense}
Lemma~\ref{lem:real1} remains valid if the partitions $\pi\in{\cal P}(a,b)$
are restricted to have nodes $t_i$ in $S\cap[a,b]$, where $S$ is an
arbitrary countable dense subset of $\R$: the estimate on $\rho_k(\pi)$
only uses that the $t_i$ lie in $[a,b]$, so it holds unchanged for this
restricted family, which is still directed as $\delta(\pi)\to0$ since $S$
is dense. This shows that the construction only relies on the density of
$S$ in $\R$, in contrast with the classical definition of the Riemann
integral for continuous functions, which invokes uniform continuity via
the compactness (hence completeness) of $[a,b]$.
\end{remark}

\subsection{Polynomials belong to the class ${\cal R}$}

In every subinterval $[t_i,t_{i+1}]$ choose an arbitrary tag
$\tau_i\in[t_i,t_{i+1}]$.

\begin{lemma}\label{lem:real2}
With the same constant $M$ as in Lemma~\ref{lem:real1},
\[
\Big|\sum_{i=0}^{N-1}\tau_i^m(t_{i+1}-t_i)-\sum_{i=0}^{N-1}t_i^m(t_{i+1}-t_i)\Big|
\le M(b-a)\,\delta(\pi).
\]
\end{lemma}
\begin{proof}
Since $\tau_i^m-t_i^m=(\tau_i-t_i)(\tau_i^{m-1}+\tau_i^{m-2}t_i+\cdots+t_i^{m-1})$
and $|\tau_i-t_i|\le t_{i+1}-t_i$, we get
\[
\Big|\sum_{i=0}^{N-1}(\tau_i^m-t_i^m)(t_{i+1}-t_i)\Big|
\le\sum_{i=0}^{N-1}\big|\tau_i^{m-1}+\cdots+t_i^{m-1}\big|(t_{i+1}-t_i)^2.
\]
As $\tau_i,t_i\in[a,b]$, the bracketed sum of $m$ terms is bounded by
$M=m(|a|+|b-a|)^{m-1}$, and the claim follows exactly as in
Lemma~\ref{lem:real1}.
\end{proof}

Lemmas~\ref{lem:real1} and~\ref{lem:real2} together show that we have
computed, directly, the Riemann integral of $t^m$:
\begin{equation}\label{real:eq0}
\int_a^b t^m\,dt=\frac{b^{m+1}-a^{m+1}}{m+1}.
\end{equation}
By linearity we obtain, for any polynomial $P$,
\begin{thm}\label{thm:real}
Let $P(t)$ be a polynomial and $P'(t)$ its derivative. Then
\[
\int_a^b P'(t)\,dt=P(b)-P(a).
\]
\end{thm}

\begin{remark}
Remark~\ref{rem:dense} applies here as well. In the next section, however,
we do need the completeness of $\R$, in order to approximate continuous
functions uniformly by polynomials.
\end{remark}

\subsection{Two immediate variants}

The same argument, applied to $\alpha_i=t_i$ with different exponents on
the two indices, gives without extra work
\[
\sum_{i=0}^{N-1}t_i^m\,(t_{i+1}^n-t_i^n)\longrightarrow
\frac{n}{m+n}\big(b^{m+n}-a^{m+n}\big),
\]
\[
\sum_{i=0}^{N-1}\frac{t_i^m}{t_{i+1}}(t_{i+1}-t_i)\longrightarrow
\frac{b^m-a^m}{m}.
\]

\begin{remark}
Choosing $a=0$, $b=1$ and the uniform partition $t_i=i/n$ in
\eqref{real:eq0}, one recovers the classical asymptotics
\[
\frac{1^m+2^m+\cdots+n^m}{n^{m+1}}\longrightarrow\frac1{m+1}.
\]
More generally, the Riemann sums for tags $\tau_i\in[t_i,t_{i+1}]$ over a
generic partition of $[a,b]$ converge if and only if the corresponding
sums over $[0,1]$ do, via the affine change of variable $t=a+s(b-a)$; so
it is enough to carry out the computation on $[0,1]$ and transport it back
to $[a,b]$.
\end{remark}

\subsection{Continuous functions belong to the class ${\cal R}$}\label{sec:cont-riemann}

It is straightforward to extend Theorem~\ref{thm:real} to piecewise
polynomials, by linearity and additivity, and to establish
\[
\Big|\int_a^b P(t)\,dt\Big|\le\int_a^b|P(t)|\,dt
\]
directly on the Riemann sums, then passing to the limit.

We now extend everything to continuous functions, via the Weierstrass
approximation theorem. Alternatively, one could prove the fundamental
theorem of calculus using the classical fact that the uniform limit of a
sequence of functions is differentiable if the sequence of derivatives
converges uniformly \cite[Theorem 7.17]{R}. (Of course, the classical
existence proof of the Riemann integral for continuous functions is much
simpler; the point here is only to keep working with the same elementary
method throughout.)

Given $\phi\colon[a,b]\to\R$ continuous, let $P_n^\phi$ be its Bernstein
polynomial of degree $n$,
\[
P_n^\phi(t)=\sum_{k=0}^n\binom nk\phi\Big(a+(b-a)\tfrac kn\Big)
\frac{(t-a)^k(b-t)^{n-k}}{(b-a)^n},
\]
so that $\|\phi-P_n^\phi\|_\infty\to0$ as $n\to\infty$, and
$\|P_n^\phi\|_\infty\le\|\phi\|_\infty$. Let
\[
v_n:=\int_a^b P_n^\phi(t)\,dt.
\]
From
\[
|v_{n+p}-v_n|\le\int_a^b|P_{n+p}^\phi(t)-P_n^\phi(t)|\,dt
\le(b-a)\,\|P_{n+p}^\phi-P_n^\phi\|_\infty
\]
the sequence $\{v_n\}$ is Cauchy, hence converges to some $v$.

\begin{remark}
The same argument with variable upper limit $t\le b$,
\[
v_n(t):=\int_a^t P_n^\phi(s)\,ds,
\]
gives, from
\[
|v_{n+p}(t)-v_n(t)|\le(b-a)\,\|P_{n+p}^\phi-P_n^\phi\|_\infty,
\]
that $\{v_n(t)\}$ converges to some $v(t)$, uniformly in $t$.
\end{remark}

Now, for a partition $\pi$ with tags at the left endpoints $t_i$,
\begin{align*}
\Big|\sum_{i=0}^{N-1}\phi(t_i)(t_{i+1}-t_i)-v\Big|
&\le\|\phi-P_n^\phi\|_\infty(b-a)\\
&\quad+\Big|\sum_{i=0}^{N-1}P_n^\phi(t_i)(t_{i+1}-t_i)-v_n\Big|+|v_n-v|,
\end{align*}
whence, letting $\delta(\pi)\to0$ first (which kills the middle term, by
Lemma~\ref{lem:real1} applied to the polynomial $P_n^\phi$) and then
$n\to\infty$,
\[
\limsup_{\delta(\pi)\to0}\Big|\sum_{i=0}^{N-1}\phi(t_i)(t_{i+1}-t_i)-v\Big|
\le\|\phi-P_n^\phi\|_\infty(b-a)+|v_n-v|\xrightarrow[n\to\infty]{}0.
\]
Hence the Riemann sums of $\phi$ converge, and we may define
$v=\int_a^b\phi(t)\,dt$; moreover
\[
\int_a^b P_n^\phi(t)\,dt\longrightarrow\int_a^b\phi(t)\,dt,\qquad
\int_a^t P_n^\phi(s)\,ds\longrightarrow\int_a^t\phi(s)\,ds\ \ \text{uniformly}.
\]
Since
\[
\frac{d}{dt}\int_a^tP_n^\phi(s)\,ds=P_n^\phi(t)\longrightarrow\phi(t)\quad\text{uniformly,}
\]
the limit $\int_a^t\phi(s)\,ds$ is differentiable, with derivative
$\phi(t)$: this is the fundamental theorem of calculus for continuous
functions.

From $|P_n^\phi(t)|\le P_n^{|\phi|}(t)$ we get
\[
\Big|\int_a^b P_n^\phi(t)\,dt\Big|\le\int_a^bP_n^{|\phi|}(t)\,dt
\longrightarrow\int_a^b|\phi(t)|\,dt,
\]
hence
\[
\Big|\int_a^b\phi(t)\,dt\Big|\le\int_a^b|\phi(t)|\,dt.
\]

\begin{cor}
If $\phi_n\to\phi$ uniformly on $[a,b]$ then
$\int_a^b\phi_n(t)\,dt\to\int_a^b\phi(t)\,dt$; indeed
\[
\Big|\int_a^b\phi_n(t)\,dt-\int_a^b\phi(t)\,dt\Big|
\le\int_a^b|\phi_n(t)-\phi(t)|\,dt\le\|\phi_n-\phi\|_\infty(b-a).
\]
\end{cor}

\subsection{The multidimensional target case}

The extension of the Riemann integral to continuous functions with values
in $\R^d$ ($d>1$) is immediate: writing $f=(f_1,\ldots,f_d)$, one
integrates componentwise,
\[
\int_a^bf(t)\,dt=\Big(\int_a^bf_1(t)\,dt,\ldots,\int_a^bf_d(t)\,dt\Big).
\]

\section{The complex case}

Let $a,b\in\C$; fix once and for all a disk of radius $R$ containing
$a,b$, and let ${\cal P}(a,b)$ be the directed set of finite sequences of
points $\pi=\{z_0=a,z_1,\ldots,z_N=b\}$ \emph{all contained in this same
fixed disk} -- $R$ is chosen once, in advance, for the whole family, not
separately for each $\pi$ -- directed by $\delta(\pi)=\max_i|z_{i+1}-z_i|$;
set also $Q(\pi)=\sum_{i=0}^{N-1}|z_{i+1}-z_i|^2$.

\begin{remark}[Why $R$ must be fixed for the whole family]\label{rem:fixed-R}
Fixing $R$ in advance, rather than allowing it to depend on $\pi$, is not
a bookkeeping nicety: it is exactly what later makes the quantitative
estimates of this section uniform along the net. In
Section~\ref{sec:tags-convention} below, integrating a holomorphic $f$
along a strange net requires a \emph{fixed} compact $K$ (contained in
this disk) with $f$ holomorphic on a neighbourhood of $K$, and the
Taylor-remainder bound there uses the constant
$M_2=\sup_{K'}|f''|<\infty$ for a \emph{fixed} enlargement $K'$ of $K$
(Lemma~\ref{lem:taylor2}) -- a constant that has no meaning, and no
reason to stay finite, if points of $\pi$ were allowed to range outside
$K'$ as $\pi$ varies. Fixing $R$ (hence $K\subset\{|z|\le R\}$) once and
for all is what guarantees every $\pi\in{\cal P}(a,b)$ stays inside $K'$,
so that the same $M_2$ bounds the remainder for every partition in the
net, uniformly. A consequence worth flagging explicitly: since every
$\pi\in{\cal P}(a,b)$ lies inside the fixed disk of radius $R$, every
point of every such $\pi$ satisfies $|z_i-a|\le2R$ -- the excursion of
any admissible strange net is automatically bounded, by construction,
regardless of how small $Q(\pi)$ is.
\end{remark}

\begin{remark}\label{rem:strange-hyp}
Throughout this section we require $Q(\pi)\to0$ as $\delta(\pi)\to0$: this
does \emph{not} follow automatically from $\delta(\pi)\to0$ alone, and it is
a genuine standing hypothesis. Under it, the points $z_i$ need not converge
to any path of finite length, or even to a path at all -- it is a
``strange'' net.
\end{remark}

\begin{remark}[$\delta(\pi)\to0$ is, conversely, a free corollary of
$Q(\pi)\to0$]\label{rem:delta-from-Q}
Since $\delta(\pi)^2=\max_i|z_{i+1}-z_i|^2$ is, trivially, one term of
the sum $Q(\pi)=\sum_i|z_{i+1}-z_i|^2$ (all of whose summands are
non-negative), we always have
\[
\delta(\pi)^2\le Q(\pi),\qquad\text{for every }\pi\in{\cal P}(a,b).
\]
Consequently $Q(\pi)\to0$ by itself already forces $\delta(\pi)\to0$: the
standing hypothesis of Remark~\ref{rem:strange-hyp} could equivalently be
stated as ``$Q(\pi)\to0$'' alone, with $\delta(\pi)\to0$ following as a
free corollary rather than as an independent requirement -- it is
$Q(\pi)$, not $\delta(\pi)$, that genuinely directs the net.

This does \emph{not}, however, make the mesh redundant in the ordering
of Remark~\ref{rem:strange-order} below: \emph{pairwise},
$Q(\pi')\le Q(\pi)$ does not imply $\delta(\pi')\le\delta(\pi)$ (a single
large gap in $\pi'$ can coexist with a small $Q(\pi')$ if compensated by
many tiny ones elsewhere, while $\pi$ has no such large gap at all), so
the two conditions defining $\pi\preceq\pi'$ are not redundant as a
comparison between individual partitions. It is only in the limit --
once $Q(\pi)\to0$ along the net -- that $\delta(\pi)\to0$ comes for free,
via the pointwise inequality above applied to the tail of the net. We
keep both conditions visible throughout: $\delta(\pi)\to0$ because it is
the classical mesh condition the reader expects of a Riemann sum, and
$Q(\pi)\to0$ because it is the genuinely new, non-automatic ingredient
specific to this construction.
\end{remark}

\subsection*{Three notions of integral}

It is worth placing the construction of this section against two more
familiar ones. Given a continuous path $\gamma\colon[0,1]\to\C$ from $a$ to
$b$: (i) the \emph{classical line integral} $\int_\gamma f(z)\,dz$ is the
limit of Riemann sums over partitions refining $\gamma$ (i.e.\ $z_i=\gamma(t_i)$,
$\max_i|t_{i+1}-t_i|\to0$), and requires $\gamma$ rectifiable; (ii) the
\emph{Young integral} relaxes this to $\gamma$ merely H\"older continuous
with exponent $\alpha>1/2$, provided the integrand is of bounded variation;
(iii) the integral of this section drops the path altogether, letting $\pi$
range over \emph{all} finite point sets in ${\cal P}(a,b)$, and asks only
that $\delta(\pi)\to0$ and $Q(\pi)\to0$.

\begin{lemma}[Directedness]\label{lem:directed}
${\cal P}(a,b)$, ordered by $\pi\preceq\pi'$ whenever
$\delta(\pi')\le\delta(\pi)$ and $Q(\pi')\le Q(\pi)$, is a directed
preorder: for any $\pi_1,\pi_2\in{\cal P}(a,b)$ there is
$\pi_3\in{\cal P}(a,b)$ with $\pi_3\succeq\pi_1$ and $\pi_3\succeq\pi_2$.
\end{lemma}
\begin{proof}
Reflexivity and transitivity of $\preceq$ are immediate from those of
$\le$ on each coordinate. For directedness, given $\pi_1,\pi_2$, let
$\pi_3$ consist of $n$ equally spaced points on the straight segment
$[a,b]\subset\C$, for $n$ to be chosen: $z_j=a+j(b-a)/n$,
$j=0,\ldots,n$. Then
\[
\delta(\pi_3)=\frac{|b-a|}n,\qquad
Q(\pi_3)=n\Big(\frac{|b-a|}n\Big)^2=\frac{|b-a|^2}n,
\]
both $\to0$ as $n\to\infty$; choosing $n$ large enough that
$\delta(\pi_3)\le\min(\delta(\pi_1),\delta(\pi_2))$ and
$Q(\pi_3)\le\min(Q(\pi_1),Q(\pi_2))$ simultaneously (possible since both
quantities $\to0$) gives $\pi_3\succeq\pi_1$ and $\pi_3\succeq\pi_2$.
Since $\pi_3\subset[a,b]$ lies in the same fixed disk of radius $R$ as
$\pi_1,\pi_2$ (Remark~\ref{rem:fixed-R}; the segment $[a,b]$ is contained
in any disk containing $a,b$), $\pi_3\in{\cal P}(a,b)$.
\end{proof}

\begin{remark}\label{rem:strange-order}
By Lemma~\ref{lem:directed}, ${\cal P}(a,b)$, ordered by
$\pi\preceq\pi'$ whenever $\delta(\pi')\le\delta(\pi)$
and $Q(\pi')\le Q(\pi)$, is a directed set, and the Riemann sums
$S(\pi)$ a net on it in the sense of Moore--Smith. We emphasize once,
here, a point implicit throughout: despite the word ``partition'',
$\preceq$ is \emph{not} a refinement order in the usual set-theoretic
sense ($\pi'\succeq\pi$ does not mean $\pi'\supseteq\pi$ as point sets);
it compares only the two numbers $\delta,Q$, and $\pi_3$ in the proof of
Lemma~\ref{lem:directed} need share no points at all with $\pi_1$ or
$\pi_2$ to dominate both of them. Since $\C$ is Hausdorff,
the limit of a convergent net is automatically unique; this guarantees, for
free, that the numbers $I_k(\pi)$ of Lemma~\ref{lem:cplx1} cannot converge
to two different values. It does \emph{not}, however, make the telescoping
argument of that Lemma superfluous: existence of the limit, and its
explicit value $L=(b^{m+1}-a^{m+1})/(m+1)$, still have to be established
directly, exactly as done there -- the net-theoretic observation only adds
uniqueness for free once convergence is known.
\end{remark}

When a partition $\pi\in{\cal P}(a,b)$ happens to refine a path $\gamma$,
the limit above agrees with the classical line integral (if $\gamma$ is
rectifiable) or with the Young integral (if $\gamma$ is H\"older with
$\alpha>1/2$, see the remark below). But ${\cal P}(a,b)$ also contains
partitions with $\delta(\pi)\to0$ and $Q(\pi)\to0$ that are not refinements
of \emph{any} continuous path: for these the limit still exists, but there
is no path for it to be associated with.

\begin{example}[An oscillating net of infinite length]\label{ex:oscillating}
On $[0,1]\subset\R$ fix $n$ and set $z_k=k/n$, $k=0,\ldots,n$. Between $z_k$
and $z_{k+1}$ insert a zig-zag of $M_n=2n^4$ steps
\[
\delta_j=\frac1{nM_n}+i\,(-1)^ja_n,\qquad j=1,\ldots,M_n,\qquad a_n=\frac1{n^3},
\]
so that the real parts sum exactly to $1/n$ and the (alternating) imaginary
parts cancel, the path returning to the real axis at $z_{k+1}$. Since
$1/(nM_n)=1/(2n^5)\ll a_n$, each step has length $|\delta_j|\sim a_n=1/n^3$,
so that, over the whole partition $\pi_n$,
\[
\delta(\pi_n)\sim\frac1{n^3}\longrightarrow0,\qquad
Q(\pi_n)\sim n\cdot M_n\,a_n^2\sim n\cdot n^4\cdot n^{-6}=\frac1n\longrightarrow0,
\]
while the total length is
\[
V(\pi_n)\sim n\cdot M_n\,a_n\sim n\cdot n^4\cdot n^{-3}=n^2\longrightarrow\infty.
\]
The polygonal paths through $\pi_n$ converge uniformly to the segment
$[0,1]$, yet their length explodes: a deterministic analogue of the fact
that Brownian motion has finite quadratic variation but infinite total
variation.
\end{example}

\begin{remark}\label{rem:holder}
Conversely, if $\gamma$ is H\"older continuous with exponent $\alpha>1/2$,
then
\[
Q(\pi)=\sum_i|\gamma(t_{i+1})-\gamma(t_i)|^2\le C\sum_i|t_{i+1}-t_i|^{2\alpha}
\longrightarrow0,
\]
since $2\alpha>1$: the standing hypothesis $Q(\pi)\to0$ is automatically
satisfied along any such path -- a class that includes paths of infinite
total variation -- so the integral of this section contains the Young
integral as a special case, while not requiring any underlying path at
all.
\end{remark}

\begin{remark}[What is new here, precisely]\label{rem:novelty}
It is worth stating explicitly, rather than leaving implicit, what kind
of contribution the construction $\mathcal P(a,b)$, $Q(\pi)\to0$ is
intended to be, since the individual results it recovers (the value of
$\int_\gamma f$, It\^o's formula, the Young--Kondurar theorem) are all
classical.

\emph{What is not new.} For a holomorphic $f$ with primitive $F$, the
limit obtained is always $F(b)-F(a)$ (up to the winding-number
correction of Theorem~\ref{thm:logwinding} when $F$ is multi-valued): the construction
never produces a numerical value not already computable, once a
primitive is known, by classical means. It is not a new integral in the
sense of assigning values to integrands for which no value was
previously defined.

\emph{What is new: an enlarged, path-free convergence class.} Classical
Riemann, Riemann--Stieltjes, and Young integration are all defined as
limits over partitions \emph{of a fixed underlying path} $\gamma$: the
points $z_i=\gamma(t_i)$ lie, by construction, on $\gamma$, and mesh
refinement means refining the parametrization of that one path. The
family $\mathcal P(a,b)$ of Section~2 drops this altogether: its
elements are arbitrary finite point sets in $\C$, required to satisfy
only the two numerical conditions $\delta(\pi)\to0$, $Q(\pi)\to0$, with
no path -- rectifiable, H\"older, or otherwise -- required to exist
behind them. Example~\ref{ex:oscillating} (and the more elaborate variants in a
companion note, available on request) exhibit nets in $\mathcal P(a,b)$
whose points do \emph{not} converge to any continuous curve at all,
along which the classical primitive formula nonetheless holds. This is
squarely type~(2)/(3) in the sense of the report -- a new convergence
framework for a classical integral, together with an elementary
unifying viewpoint on why several classical theorems (Riemann,
Riemann--Stieltjes, Young--Kondurar, and their complex analogues) share
one proof -- and \emph{not} type~(1): we make no claim to a new class of
integrable functions or new values.

\emph{Relation to Young integration and the Sewing Lemma.}
Remark~2.5 above already shows that $\alpha$-H\"older paths with
$\alpha>1/2$ satisfy $Q(\pi)\to0$ automatically, so every instance of
the classical Young integral is an instance of this construction. The
Sewing Lemma of Feyel--de la Pradelle~\cite{FdlP} (see also the
discussion around Theorem~4.5 in Section~4.2) operates at a different,
complementary level of generality: it constructs an integral for a
possibly irregular integrand against a possibly irregular path, given
only a local quasi-additivity estimate on a germ -- the path, however
rough, is still the primary object. Here the integrand is always
classical (a polynomial, or a holomorphic $f$ with known local
behaviour), and what is relaxed instead is the requirement that the
approximating points lie along \emph{any} path whatsoever. In this
sense the two generalizations are orthogonal: the Sewing Lemma weakens
the regularity of the path while keeping the notion of ``a path'';
$\mathcal P(a,b)$ keeps the regularity of the integrand high while
discarding the path altogether. Modern rough-path theory
\cite{Lyons98,FV} addresses a further, harder question that the present
construction does not touch at all: how to \emph{define} the integral
(and iterated integrals) of one rough signal against another, where no
classical pairing exists and an enriched path (the rough-path lift)
must be supplied as extra data. Nothing in this paper bears on that
problem; the present construction only ever integrates a fixed,
classically integrable $f$, and the question it asks is how weak a
notion of ``approximating point set'' still recovers the classical
answer.

\emph{Relation to F\"ollmer's pathwise It\^o calculus.} A third,
independent point of comparison, orthogonal to both of the above, is
F\"ollmer's pathwise proof of It\^o's formula~\cite{F} and its
far-reaching generalization to paths of arbitrary $p$-th variation by
Cont and Perkowski~\cite{ContPerkowski19}. That theory fixes a single
path $S$ and a sequence of partitions refining it, and constructs, for
$S$ with finite $p$-th variation $[S]^p\not\equiv0$ along that sequence,
a compensated Riemann sum converging to
$\int_0^tf'(S)\,dS+\frac1{p!}\int_0^tf^{(p)}(S)\,d[S]^p$ -- the general
Taylor-remainder mechanism used throughout the present paper, but with
$[S]^p$ left as an \emph{arbitrary} nondecreasing function, rather than
required to vanish. Our Section~2 is exactly the degenerate case
$[S]^2\equiv0$ of that formula (whence no correction term survives),
and our Example~2.6 is a deterministic instance of the same phenomenon
\cite[Remark~1.2]{ContPerkowski19} note for Brownian motion: finite
(here, zero) $p$-th variation along a net or sequence does not imply
finite $p$-variation in the classical (Wiener) sense, so a path can
satisfy $Q(\pi_n)\to0$ while having infinite total variation. Their
extension to vector-valued paths, in F\"ollmer's original formulation,
requires $S^i,S^i+S^j\in V^p(\pi)$ for each pair of components -- the
same polarization device used below (Lemma~\ref{lem:bilinear}) to
construct $\langle A,B\rangle$ from $\langle A\rangle,\langle
B\rangle,\langle A+B\rangle$ for conformal martingales, applied here to
the pair $(A,B)$ making up $M=A+iB$. Where F\"ollmer and Cont--Perkowski
generalize by weakening the regularity of a fixed path (allowing
$[S]^p\ne0$, and $p$ any even integer, to reach paths as rough as
fractional Brownian motion of arbitrary Hurst index), the present paper
fixes $p=2$ throughout and generalizes in the orthogonal direction, by
discarding the path itself.

\emph{Genuine uses of the path-free formulation.} We do not currently
have an application in which dropping the underlying path is
\emph{essential} -- as opposed to illustrative -- rather than merely an
equivalent reformulation of a path-based statement. We record this
honestly rather than overstate it: the most plausible candidates, not
pursued here, are situations where the approximating points arise from
a source with no natural path structure to begin with (a scattered or
random point set, e.g.\ from a numerical scheme or a point process),
where knowing that $Q(\pi)\to0$ alone -- without having to exhibit or
verify convergence to any particular curve -- suffices to guarantee the
classical formula. Whether this is more than a reformulation is, in our
view, an open question, and one the present paper does not resolve.
\end{remark}

\subsection{The key result}

Given $\pi\in{\cal P}(a,b)$ and a fixed integer $m$, define, as before,
\[
I_k(\pi)=\ds\sum_{i=0}^{N-1}z_i^{m-k}z_{i+1}^k(z_{i+1}-z_i),\qquad k=0,\ldots,m.
\]

\begin{lemma}\label{lem:cplx1}
If $Q(\pi)\to0$ as $\delta(\pi)\to0$, all the $I_k(\pi)$ tend to the same
limit $L=\dfrac{b^{m+1}-a^{m+1}}{m+1}$.
\end{lemma}
\begin{proof}
Exactly as in Lemma~\ref{lem:real1}, writing $I_k(\pi)=I_0(\pi)+\rho_k(\pi)$
for $1\le k\le m$, the same telescoping identity gives
\[
(m+1)I_0(\pi)+\sum_{k=1}^m\rho_k(\pi)=b^{m+1}-a^{m+1},
\]
with
\[
\rho_k(\pi)=\sum_{i=0}^{N-1}z_i^{m-k}\big(z_{i+1}^{k-1}+z_{i+1}^{k-2}z_i+\cdots+z_i^{k-1}\big)(z_{i+1}-z_i)^2,
\]
and, with $M=mR^{m-1}$,
\[
|\rho_k(\pi)|\le M\sum_{i=0}^{N-1}|z_{i+1}-z_i|^2=M\,Q(\pi).
\]
Hence $\rho_k(\pi)\to0$ as soon as $Q(\pi)\to0$, and every
$I_k(\pi)$ converges to $L=\frac{b^{m+1}-a^{m+1}}{m+1}$.
\end{proof}

\subsection{Polynomials belong to the class ${\cal R}$}

For each $i$ choose an arbitrary tag $\zeta_i$ with
$|\zeta_i-z_i|\le|z_{i+1}-z_i|$ (any $\zeta_i\in[z_i,z_{i+1}]$ qualifies,
but the tag need not lie on that segment -- nothing below uses
collinearity, only closeness to $z_i$).

\begin{lemma}\label{lem:cplx2}
With the same $M=mR^{m-1}$,
\[
\Big|\sum_{i=0}^{N-1}\zeta_i^m(z_{i+1}-z_i)-\sum_{i=0}^{N-1}z_i^m(z_{i+1}-z_i)\Big|
\le M\,Q(\pi).
\]
\end{lemma}
\begin{proof}
Exactly as in Lemma~\ref{lem:real2}, using
$\zeta_i^m-z_i^m=(\zeta_i-z_i)(\zeta_i^{m-1}+\cdots+z_i^{m-1})$ and
$|\zeta_i-z_i|\le|z_{i+1}-z_i|$.
\end{proof}

Lemmas~\ref{lem:cplx1} and~\ref{lem:cplx2} show that we have computed the
Riemann integral of $z^m$ along the path $\pi$:
\[
\int_a^b z^m\,dz=\frac{b^{m+1}-a^{m+1}}{m+1},
\]
and, by linearity,
\begin{thm}\label{thm:cplx}
Let $P(z)$ be a polynomial and $P'(z)$ its derivative. Then
\[
\int_a^b P'(z)\,dz=P(b)-P(a).
\]
\end{thm}

By Remark~\ref{rem:holder} above, this extends the classical Cauchy
theorem for polynomials to any H\"older path with exponent $\alpha>1/2$,
recovering in particular the setting of Young's integral.

\subsection{Cauchy's theorem along strange nets}\label{sec:tags-convention}

We now extend Theorem~\ref{thm:cplx} from polynomials to arbitrary
holomorphic functions, still along any strange net with $Q(\pi)\to0$. We
state it directly for a general compact set $K$ rather than a disk
centered at $0$: this costs nothing in the proof, and is what is actually
needed below, since the natural domains for negative and fractional
powers of $z$ are punctured disks and annuli, which are not convex and do
not contain the origin.

Let $D\subset\C$ be open and $f$ holomorphic on $D$, and let $K\subset D$
be compact, with $d:=\operatorname{dist}(K,\partial D)>0$ (set $d=+\infty$
if $D=\C$). Fix $d'\in(0,d)$ and let
$K':=\{z\in\C:\operatorname{dist}(z,K)\le d'\}$, a compact subset of $D$
(since every $z\in K'$ satisfies $\operatorname{dist}(z,\partial D)\ge
d-d'>0$). Since $f''$ is continuous on $D$, it is bounded on $K'$; set
\[
M_2:=\sup_{K'}|f''|<\infty.
\]

\begin{lemma}[Second-order Taylor remainder]\label{lem:taylor2}
For every $z\in K$ and $w\in\C$ with $|w-z|\le d'$, the segment $[z,w]$ is
contained in $K'$ (hence in $D$), and
\[
\big|f(w)-f(z)-f'(z)(w-z)\big|\le\frac{M_2}2|w-z|^2.
\]
\end{lemma}
\begin{proof}
Every point $u=z+t(w-z)$, $t\in[0,1]$, of the segment satisfies
$\operatorname{dist}(u,K)\le|u-z|\le|w-z|\le d'$, so $u\in K'\subset D$;
in particular the segment is contained in $D$. Writing
$\gamma(t)=z+t(w-z)$, the real fundamental theorem of calculus applied to
$t\mapsto f(\gamma(t))$ (which is $C^1$ in $t$, since $\gamma([0,1])\subset D$) gives
\[
f(w)-f(z)=\int_0^1f'(z+t(w-z))(w-z)\,dt.
\]
Subtracting $f'(z)(w-z)$,
\[
f(w)-f(z)-f'(z)(w-z)=(w-z)\int_0^1\big[f'(z+t(w-z))-f'(z)\big]\,dt.
\]
Applying the same identity to $f'$ (with $f''$ in place of $f'$, still
along the same segment $\subset K'$),
\[
f'(z+t(w-z))-f'(z)=\int_0^tf''(z+s(w-z))(w-z)\,ds,
\]
\[
|f'(z+t(w-z))-f'(z)|\le M_2\,t\,|w-z|.
\]
Hence
\[
|f(w)-f(z)-f'(z)(w-z)|\le|w-z|\int_0^1M_2\,t\,|w-z|\,dt=\frac{M_2}2|w-z|^2.
\qedhere
\]
\end{proof}

\begin{thm}[Cauchy's theorem along strange nets]\label{thm:cauchynet}
Let $a,b\in K$ and let $\pi\in{\cal P}(a,b)$ be a partition with all nodes
$z_i\in K$ and arbitrary tags $\zeta_i$ with $|\zeta_i-z_i|\le|z_{i+1}-z_i|$, with the standing
hypothesis $Q(\pi)\to0$ as $\delta(\pi)\to0$ (and, eventually,
$\delta(\pi)\le d'$). If $f$ is holomorphic on $D\supset K$ then
\[
S(\pi):=\sum_{i=0}^{N-1}f'(\zeta_i)(z_{i+1}-z_i)\longrightarrow f(b)-f(a).
\]
\end{thm}
\begin{proof}
\emph{Step 1} (tags $=z_i$). From $f(b)-f(a)=\sum_i[f(z_{i+1})-f(z_i)]$ and
Lemma~\ref{lem:taylor2} applied to each increment (valid once
$\delta(\pi)\le d'$, since $z_i,z_{i+1}\in K$),
\[
\Big|f(b)-f(a)-\sum_if'(z_i)(z_{i+1}-z_i)\Big|\le\sum_i\frac{M_2}2|z_{i+1}-z_i|^2
=\frac{M_2}2Q(\pi).
\]

\emph{Step 2} (arbitrary tags $\zeta_i$). Applying Lemma~\ref{lem:taylor2}
in first-order form to $f'$ (whose derivative $f''$ is bounded by $M_2$ on
$K'$) gives $|f'(\zeta_i)-f'(z_i)|\le M_2|\zeta_i-z_i|\le M_2|z_{i+1}-z_i|$,
hence
\[
\Big|\sum_if'(\zeta_i)(z_{i+1}-z_i)-\sum_if'(z_i)(z_{i+1}-z_i)\Big|
\le\sum_iM_2|z_{i+1}-z_i|^2=M_2\,Q(\pi).
\]

\emph{Step 3.} By the triangle inequality,
\[
\big|S(\pi)-(f(b)-f(a))\big|\le\frac{3M_2}2\,Q(\pi)
\xrightarrow[\delta(\pi)\to0]{}0. \qedhere
\]
\end{proof}

\begin{cor}[Classical Cauchy theorem]\label{cor:cauchyclassic}
If $g$ is holomorphic on a simply connected open set $D$ and $K\subset D$
is compact, $g$ admits a primitive $f$ on $D$ (i.e.\ $f'=g$);
Theorem~\ref{thm:cauchynet} then gives
\[
\int_a^bg(z)\,dz=f(b)-f(a),
\]
independently of the strange net chosen to join $a$ to $b$ inside $K$; in
particular, for $a=b$,
\[
\oint_\gamma g(z)\,dz=0,
\]
the classical Cauchy theorem -- now valid along any net with
$Q(\pi)\to0$, including paths of infinite total variation (e.g.
$\alpha$-H\"older with $1/2<\alpha<1$), not only rectifiable ones.
\end{cor}

\subsection{Negative powers}

We now compute $\int z^{-m}\,dz$ for an integer $m\ge1$, along a net
avoiding the origin. The cases $m\ge2$ and $m=1$ behave very differently:
$\operatorname{Res}_0(z^{-m})=0$ for $m\ge2$, while
$\operatorname{Res}_0(z^{-1})=1$.

\subsubsection{The case $m\ge2$}

For $m\ge2$, the function $F(z)=z^{1-m}/(1-m)$ is a single-valued
primitive of $z^{-m}$ on $D=\C\setminus\{0\}$ (there is no branch point,
since $1-m\le-1$ is an integer), so Theorem~\ref{thm:cauchynet} applies
directly -- with $K=A$ a compact annulus, exactly the setting for which
that theorem was stated -- on any compact annulus $A=\{r\le|z|\le R\}\subset D$
containing the net.

\begin{cor}\label{cor:negpower}
Let $\pi\in{\cal P}(a,b)$ with $a,b\in D$ and all nodes and tags in a
compact annulus $A\subset D$, and let $m\ge2$. Then, under the standing
hypothesis $Q(\pi)\to0$,
\[
\sum_{i=0}^{N-1}\zeta_i^{-m}(z_{i+1}-z_i)\longrightarrow
\int_a^bz^{-m}\,dz=\frac{b^{1-m}-a^{1-m}}{1-m},
\]
independently of the strange net chosen to join $a$ to $b$ inside $A$ --
in particular, independently of how many times the net winds around the
origin.
\end{cor}
\begin{proof}
Immediate from Theorem~\ref{thm:cauchynet} applied to $f=F$ on $D=\C\setminus\{0\}$
with $K=A$, since $F'=z^{-m}$ there.
\end{proof}

\subsubsection{The case $m=1$}\label{sec:argument-m1}

Here the only local primitive is $\mathrm{Log}\,z$, which is not
single-valued on $\C\setminus\{0\}$: its value jumps by $2\pi i$ upon
winding around the origin. Throughout, $\arg_0(z)\in(-\pi,\pi]$ denotes
the principal argument of $z\in\C\setminus\{0\}$, and
$\mathrm{Log}\,z:=\ln|z|+i\arg_0(z)$ the principal branch of the
logarithm; more generally, for $w\ne0$ and any exponent $\lambda\in\C$,
$(w^\lambda)_0:=|w|^\lambda e^{i\lambda\arg_0(w)}$ denotes the principal
branch of $w^\lambda$, evaluated with this same choice of $\arg_0$.
Consequently ${\cal P}(a,b)$, as a family of
arbitrary point sets with $\delta(\pi)\to0$, is no longer Cauchy for
$g(z)=1/z$: two admissible nets joining $a$ to $b$ but winding differently
around $0$ may yield limits differing by a multiple of $2\pi i$. We must
therefore fix a continuous path and let the net refine along it.

Let $\gamma:[0,1]\to D=\C\setminus\{0\}$ be continuous, $\gamma(0)=a$,
$\gamma(1)=b$, with $\gamma([0,1])\subset A=\{r\le|z|\le R\}$. Since
$\gamma$ is continuous and non-vanishing, it admits a continuous argument
$\theta:[0,1]\to\R$ with $\gamma(t)=|\gamma(t)|e^{i\theta(t)}$; the
\emph{total turning} of $\gamma$ is $\Theta_\gamma:=\theta(1)-\theta(0)$,
and $\Theta_\gamma=\arg_0(b)-\arg_0(a)+2\pi n(\gamma)$ for a unique winding
number $n(\gamma)\in\Z$. Partitions are now taken along $\gamma$:
$z_k=\gamma(t_k)$ for $0=t_0<\cdots<t_N=1$, with tags $\zeta_k$ such that
$|\zeta_k-z_k|\le|z_{k+1}-z_k|$ (as in Section~\ref{sec:tags-convention}), and $\delta(\pi)$,
$Q(\pi)$ as before.

\begin{lemma}\label{lem:logtel}
For $\delta(\pi)$ small enough (depending only on the modulus of
continuity of $\theta$ on $[0,1]$),
\[
\sum_{k=0}^{N-1}\mathrm{Log}\Big(\frac{z_{k+1}}{z_k}\Big)
=\ln\Big|\frac ba\Big|+i\,\Theta_\gamma,
\]
where $\mathrm{Log}$ denotes the principal branch.
\end{lemma}
\begin{proof}
By uniform continuity of $\theta$, choose $\delta(\pi)$ small enough that
$|\theta(t_{k+1})-\theta(t_k)|<\pi$ for every $k$. Then
$z_{k+1}/z_k=(|z_{k+1}|/|z_k|)e^{i(\theta(t_{k+1})-\theta(t_k))}$, and
since the exponent has modulus $<\pi$, the principal branch gives exactly
\[
\mathrm{Log}(z_{k+1}/z_k)=\big[\ln|z_{k+1}|-\ln|z_k|\big]
+i\big[\theta(t_{k+1})-\theta(t_k)\big]
\]
(no $2\pi i$ jump). Summing over $k$, the real part telescopes to
$\ln|b|-\ln|a|$ and the imaginary part to $\theta(1)-\theta(0)=\Theta_\gamma$.
\end{proof}

\begin{lemma}\label{lem:logest}
For $z,\zeta\in\C$ with $|z|\ge r$, $|\zeta-z|\le|w-z|\le r/2$, and
$w\in\C$ with $|w|\ge r$,
\[
\big|\zeta^{-1}(w-z)-\mathrm{Log}(w/z)\big|\le\frac C{r^2}\,|w-z|^2
\]
for an absolute constant $C$.
\end{lemma}
\begin{proof}
Set $u=(w-z)/z$, so $|u|\le\delta(\pi)/r\le1/2$. By the Taylor expansion
of the logarithm, $|\mathrm{Log}(1+u)-u|\le C_1|u|^2$, hence
$|\mathrm{Log}(w/z)-(w-z)/z|\le C_1|w-z|^2/r^2$. On the other hand, since
$\zeta$ need not lie on the segment $[z,w]$ -- only
$|\zeta-z|\le|w-z|\le r/2$ is guaranteed -- we only get
$|\zeta|\ge|z|-|\zeta-z|\ge r-r/2=r/2$, so
\[
\big|\zeta^{-1}(w-z)-z^{-1}(w-z)\big|=|w-z|\,\frac{|\zeta-z|}{|\zeta||z|}
\le\frac{|w-z|^2}{(r/2)\,r}=\frac{2|w-z|^2}{r^2}
\]
using $|\zeta-z|\le|w-z|$, $|\zeta|\ge r/2$, $|z|\ge r$. The triangle
inequality gives the claim with $C=C_1+2$.
\end{proof}

\begin{thm}\label{thm:logwinding}
With the notation above,
\begin{align*}
S(\pi):=\sum_{k=0}^{N-1}\zeta_k^{-1}(z_{k+1}-z_k)&\longrightarrow
\ln\Big|\frac ba\Big|+i\,\Theta_\gamma\\
&=:\int_\gamma\frac{dz}z=\mathrm{Log}(b)-\mathrm{Log}(a)+2\pi i\,n(\gamma),
\end{align*}
with the quantitative estimate
\[
\Big|S(\pi)-\Big(\ln\big|\tfrac ba\big|+i\Theta_\gamma\Big)\Big|
\le\frac C{r^2}\,Q(\pi).
\]
\end{thm}
\begin{proof}
By Lemma~\ref{lem:logest} applied to each increment and summed,
\[
\Big|S(\pi)-\sum_k\mathrm{Log}(z_{k+1}/z_k)\Big|
\le\frac C{r^2}\sum_k|z_{k+1}-z_k|^2=\frac C{r^2}Q(\pi).
\]
Combining with Lemma~\ref{lem:logtel} (valid once $\delta(\pi)$ is small),
$\sum_k\mathrm{Log}(z_{k+1}/z_k)=\ln|b/a|+i\Theta_\gamma$ exactly, whence
the claim. \qedhere
\end{proof}

\begin{remark}
The winding number $n(\gamma)\in\Z$ depends on the homotopy class of the
path in $\C\setminus\{0\}$ rel endpoints, not merely on $a,b$: paths that
wind differently around the origin yield integrals differing by
multiples of $2\pi i$. This is the residue $\operatorname{Res}_0(1/z)=1$
appearing once per winding, and stands in sharp contrast with the case
$m\ge2$, where $\operatorname{Res}_0(z^{-m})=0$ and the primitive is
single-valued, so the integral sees nothing of the shape of the net. In
this sense the present framework recovers the contribution of the
residue $1$ of the simple pole at $0$, through the winding number.
\end{remark}

\subsection{Fractional and complex powers}

We finally treat $z^s$ for $s\in\C\setminus\Z$, which is multi-valued: the
right setting is the Riemann surface of the logarithm. Let
\[
\widetilde D:=\{(\rho,\theta):\rho>0,\ \theta\in\R\},\qquad
p(\rho,\theta):=\rho e^{i\theta},
\]
the universal cover of $D=\C\setminus\{0\}$. On $\widetilde D$ the function
$\widehat{z^s}(\rho,\theta):=\rho^se^{is\theta}$ is genuinely
single-valued. Given $\gamma\colon[0,1]\to D$ continuous with
$\gamma(0)=a$, $\gamma(1)=b$, $\gamma([0,1])\subset A=\{r\le|z|\le R\}$,
and $\theta(t)$ its continuous argument as in Section~\ref{sec:argument-m1} (with
$\theta(0)=\arg_0(a)$), the lift $\widetilde\gamma(t)=(\rho(t),\theta(t))$,
$\rho(t)=|\gamma(t)|$, satisfies $p\circ\widetilde\gamma=\gamma$. Since
$[0,1]$ is compact and $\theta$ continuous, $\theta([0,1])\subset
[\theta_-,\theta_+]$ for some finite $\theta_\pm$.

For $\delta(\pi)<r/2$, the same distance argument as in
Lemma~\ref{lem:taylor2} shows that $z_{k+1}$ and any tag $\zeta_k$ with
$|\zeta_k-z_k|\le|z_{k+1}-z_k|$ lie within a slit disk around $z_k$
avoiding the origin, on which the branch of $z^{s+1}$ obtained by
continuation along $\widetilde\gamma$ from $z_k$ is holomorphic; set
\[
M_s:=|s(s+1)|\sup_{\rho\in[r/2,2R],\ \theta\in[\theta_--\pi/4,\theta_++\pi/4]}
\rho^{\Re s-1}e^{-\Im s\,\theta}<\infty,
\]
a single constant valid uniformly in $k$, finite by compactness of the
rectangle above (this is the analogue of $M_2$ for $f(z)=z^{s+1}$, since
$f''(z)=s(s+1)z^{s-1}$ and $|z^{s-1}|=\rho^{\Re s-1}e^{-\Im s\,\theta}$ on
the branch determined by $\widetilde\gamma$). Set
$Z(t):=\widehat{z^{s+1}}(\rho(t),\theta(t))$, continuous and single-valued
on $[0,1]$.

\begin{thm}\label{thm:fracpower}
Let $s\in\C\setminus\{-1\}$. With the notation above, tags $\zeta_k$ with
$|\zeta_k-z_k|\le|z_{k+1}-z_k|$, and the standing hypothesis $Q(\pi)\to0$,
\[
S(\pi):=\sum_{k=0}^{N-1}\zeta_k^s(z_{k+1}-z_k)\longrightarrow
\frac1{s+1}\Big[(b^{s+1})_0\,e^{2\pi i(s+1)n(\gamma)}-(a^{s+1})_0\Big]
\]
\[
=:\int_\gamma z^s\,dz,
\]
where $(\cdot)_0$ denotes the principal branch and $n(\gamma)\in\Z$ is the
winding number of $\gamma$ around $0$, with
\[
\Big|S(\pi)-\int_\gamma z^s\,dz\Big|\le\frac{M_s}2\,Q(\pi).
\]
\end{thm}
\begin{proof}
By construction, $Z(1)-Z(0)=\sum_k[Z(t_{k+1})-Z(t_k)]$ exactly, with
$Z(1)=(b^{s+1})_0e^{2\pi i(s+1)n(\gamma)}$ and $Z(0)=(a^{s+1})_0$, since
$\theta(1)=\arg_0(b)+2\pi n(\gamma)$. Arguing exactly as in the proof of
Lemma~\ref{lem:taylor2}, applied to the local branch of $z^{s+1}$ around
each $z_k$,
\[
\big|Z(t_{k+1})-Z(t_k)-(s+1)\zeta_k^s(z_{k+1}-z_k)\big|
\le\frac{M_s}2|z_{k+1}-z_k|^2,
\]
and summing over $k$,
\[
\big|(s+1)S(\pi)-(Z(1)-Z(0))\big|\le\frac{M_s}2\,Q(\pi),
\]
whence the claim upon dividing by $|s+1|$. \qedhere
\end{proof}

\begin{remark}
Theorem~\ref{thm:fracpower} contains the previous results as special
cases. If $s=m-1$ for an integer $m\ge1$, then $s+1=m\in\Z$, so
$e^{2\pi i(s+1)n(\gamma)}=1$ and no dependence on the winding survives:
this recovers Theorem~\ref{thm:cplx}. If $s=-m$ for an integer $m\ge2$,
then $s+1=1-m\in\Z$ and again the exponential factor is $1$: this
recovers Corollary~\ref{cor:negpower}. The case $s=-1$ is excluded here
because $s+1=0$; it is the degenerate case treated separately in
Theorem~\ref{thm:logwinding}, and is recovered in the limit $s\to-1$ by
$\lim_{s\to-1}(z^{s+1}-1)/(s+1)=\log z$ (l'H\^opital). For $s$ genuinely
fractional or non-real, the factor $e^{2\pi i(s+1)n(\gamma)}\ne1$ in
general: the value of the integral depends essentially on the homotopy
class of the path in $\C\setminus\{0\}$, i.e.\ on which sheet of the
Riemann surface of $z^s$ the path traverses.
\end{remark}

\part{Deterministic extensions}

\section{Multidimensional Riemann integral}

We proceed by induction, starting from the two-dimensional case.

\subsection{The bidimensional case}

Let $Q=[a,b]\times[c,d]$ and consider the monomial $x^ny^m$. Given a
partition of $Q$,
\[
\{x_0=a<\cdots<x_N=b\}\times\{y_0=c<\cdots<y_M=d\},
\]
and $0\le h\le n$, $0\le k\le m$, form
\[
S_{h,k}=\sum_{i=0}^{N-1}\sum_{j=0}^{M-1}
x_i^{n-h}x_{i+1}^hy_j^{m-k}y_{j+1}^k(x_{i+1}-x_i)(y_{j+1}-y_j).
\]
Applying the one-dimensional argument in each variable separately,
\[
S_{h,k}\longrightarrow\frac{b^{n+1}-a^{n+1}}{n+1}\,\frac{d^{m+1}-c^{m+1}}{m+1}
=:\int_Qx^ny^m\,dx\,dy.
\]
With $F(x,y)=\dfrac{x^{n+1}}{n+1}\dfrac{y^{m+1}}{m+1}$ this reads
\[
\int_Qx^ny^m\,dx\,dy=\int_Q\frac{\partial^2F}{\partial x\,\partial y}\,dx\,dy
=F(b,d)+F(a,c)-F(a,d)-F(b,c),
\]
and, by linearity, for every polynomial $P(x,y)$,
\[
\int_Q\frac{\partial^2P}{\partial x\,\partial y}\,dx\,dy
=P(b,d)+P(a,c)-P(a,d)-P(b,c).
\]
This extends to a finite union of rectangles $A$:
\[
\int_A\frac{\partial^2P}{\partial x\,\partial y}\,dx\,dy
=\sum_{\alpha\in\partial A}(-1)^{\operatorname{sign}\alpha}P(\alpha)
=\oint_{\partial A}\frac{\partial P}{\partial x}\,dx
=-\oint_{\partial A}\frac{\partial P}{\partial y}\,dy,
\]
and, writing $\tilde P=\partial P/\partial x$ or $\tilde P=\partial P/\partial y$ respectively,
\[
\int_A\frac{\partial\tilde P}{\partial y}\,dx\,dy=\oint_{\partial A}\tilde P\,dx,
\qquad
\int_A\frac{\partial\tilde P}{\partial x}\,dx\,dy=-\oint_{\partial A}\tilde P\,dy.
\]

\subsubsection*{Extension to general domains}

Throughout this part we assume $a,b\colon[c,d]\to\R$ are continuous, with
$a(y)<b(y)$ for every $y\in[c,d]$; no further regularity (bounded
variation, Lipschitz, or otherwise) is imposed, as the argument below
will show none is needed. Consider domains of the type
\[
D=\{(x,y)\mid a(y)\le x\le b(y)\}\ \subset\subset\ [\min_ya(y),\max_yb(y)]\times[c,d].
\]
Let $F(x,y)=x^my^n$, $G(x,y)=\dfrac{x^{m+1}}{m+1}\dfrac{y^{n+1}}{n+1}$, and
for a rectangle $\mathcal R(x_1,y_1;x_2,y_2)=[x_1,x_2]\times[y_1,y_2]$ the
previous result gives
\[
I\big(\mathcal R(x_1,y_1;x_2,y_2),F\big)=G(x_1,y_1)+G(x_2,y_2)-G(x_2,y_1)-G(x_1,y_2).
\]
Given a partition $c=y_0<y_1<\cdots<y_N=d$ of $[c,d]$, summing over the
rectangles $\mathcal R(a(y_k),y_k;b(y_k),y_{k+1})$ and using
\[
G(a(y_k),y_k)-G(b(y_k),y_k)=-\Big[\frac{b(y_k)^{m+1}}{m+1}-\frac{a(y_k)^{m+1}}{m+1}\Big]\frac{y_k^{n+1}}{n+1},
\]
\[
G(b(y_k),y_{k+1})-G(a(y_k),y_{k+1})=\Big[\frac{b(y_k)^{m+1}}{m+1}-\frac{a(y_k)^{m+1}}{m+1}\Big]\frac{y_{k+1}^{n+1}}{n+1},
\]
we obtain
\begin{align*}
\sum_k I\big(\mathcal R(a(y_k),y_k;b(y_k),y_{k+1}),F\big)
&=\sum_k\Big[\frac{b(y_k)^{m+1}}{m+1}-\frac{a(y_k)^{m+1}}{m+1}\Big]\\
&\qquad\Big[\frac{y_{k+1}^{n+1}}{n+1}-\frac{y_k^{n+1}}{n+1}\Big],
\end{align*}
which, as $\delta(\pi)\to0$, converges to
\[
\int_c^d\Big[\frac{b(y)^{m+1}}{m+1}-\frac{a(y)^{m+1}}{m+1}\Big]y^n\,dy.
\]
To see this, set $\phi(y):=\dfrac{b(y)^{m+1}-a(y)^{m+1}}{m+1}$ (continuous
on $[c,d]$, since $a,b$ are) and $\psi(y):=\dfrac{y^{n+1}}{n+1}$, so the
sum above is $\sum_k\phi(y_k)\big[\psi(y_{k+1})-\psi(y_k)\big]$. By the
mean value theorem, $\psi(y_{k+1})-\psi(y_k)=\eta_k^n(y_{k+1}-y_k)$ for
some $\eta_k\in(y_k,y_{k+1})$, so
\begin{align*}
\sum_k\phi(y_k)\big[\psi(y_{k+1})-\psi(y_k)\big]&-\sum_k\phi(y_k)\,y_k^n(y_{k+1}-y_k)\\
&=\sum_k\phi(y_k)\big[\eta_k^n-y_k^n\big](y_{k+1}-y_k),
\end{align*}
which is bounded in modulus by
$(d-c)\,\|\phi\|_\infty\sup_k|\eta_k^n-y_k^n|\to0$ as $\delta(\pi)\to0$,
since $|\eta_k-y_k|\le\delta(\pi)$ and $y\mapsto y^n$ is uniformly
continuous on the compact $[c,d]$. The remaining sum
$\sum_k\phi(y_k)y_k^n(y_{k+1}-y_k)$ is an ordinary Riemann sum of the
continuous function $\phi(y)y^n$, so it converges, by
Section~\ref{sec:cont-riemann}, to $\int_c^d\phi(y)y^n\,dy$ as
$\delta(\pi)\to0$ --- proving the claimed limit, using only the
continuity of $a,b$.
Hence, for $P(x,y)=\sum_ic_i\,\dfrac{x^{m_i+1}}{m_i+1}\,y^{n_i}$,
\[
\oint_{\partial D}P(x,y)\,dy=\int_D\frac{\partial P(x,y)}{\partial x}\,dx\,dy.
\]

To extend this equality to a general $C^1$ function $f$, uniform
approximation of $f$ by polynomials is not enough on its own: uniform
convergence $P_n\to f$ does \emph{not} imply $\partial_xP_n\to\partial_xf$
(a classical warning --- e.g.\ Weierstrass-type examples show a uniformly
convergent sequence of smooth functions whose derivatives fail to
converge at all). What is needed is \emph{simultaneous} approximation of
$f$ together with its first partial derivatives.

\begin{lemma}[Simultaneous $C^1$-approximation, one variable]\label{lem:c1-approx-1d}
Let $g\in C^1[0,1]$. The Bernstein polynomials
$B_ng(x):=\sum_{k=0}^ng(k/n)\binom nkx^k(1-x)^{n-k}$ satisfy
$B_ng\to g$ and $(B_ng)'\to g'$, both uniformly on $[0,1]$.
\end{lemma}
\begin{proof}
$B_ng\to g$ uniformly is the classical Bernstein--Weierstrass theorem.
For the derivative, differentiating the explicit formula gives the
standard identity
\[
(B_ng)'(x)=n\sum_{k=0}^{n-1}\Big[g\big(\tfrac{k+1}n\big)-g\big(\tfrac kn\big)\Big]\binom{n-1}k x^k(1-x)^{n-1-k}.
\]
By the mean value theorem, $g(\tfrac{k+1}n)-g(\tfrac kn)=\tfrac1ng'(\xi_k)$
for some $\xi_k\in(\tfrac kn,\tfrac{k+1}n)$, so
\[
(B_ng)'(x)=\sum_{k=0}^{n-1}g'(\xi_k)\binom{n-1}kx^k(1-x)^{n-1-k}.
\]
Since the Bernstein basis functions $\binom{n-1}kx^k(1-x)^{n-1-k}$ sum to
$1$ (a partition of unity) and $|\xi_k-\tfrac kn|\le\tfrac1n$,
\[
\Big|(B_ng)'(x)-\sum_{k=0}^{n-1}g'\big(\tfrac kn\big)\binom{n-1}kx^k(1-x)^{n-1-k}\Big|
\le\max_k\Big|g'(\xi_k)-g'\big(\tfrac kn\big)\Big|
\]
\[
\le\omega_{g'}\big(\tfrac1n\big),
\]
where $\omega_{g'}$ is the modulus of continuity of $g'$ (continuous,
hence uniformly continuous, on $[0,1]$); the right-hand side $\to0$ as
$n\to\infty$. The remaining sum is $B_{n-1}(g')(x)$, which $\to g'(x)$
uniformly by Bernstein--Weierstrass applied to the continuous function
$g'$. Hence $(B_ng)'\to g'$ uniformly.
\end{proof}

\begin{lemma}[Simultaneous $C^1$-approximation, two variables]\label{lem:c1-approx-2d}
Let $R=[0,1]^2$ and $f\in C^1(R)$ (i.e.\ $f,f_x,f_y$ continuous on $R$).
The tensor Bernstein polynomials
\[
B_nf(x,y):=\sum_{j=0}^n\sum_{k=0}^nf\big(\tfrac jn,\tfrac kn\big)
\binom nj x^j(1-x)^{n-j}\binom nky^k(1-y)^{n-k}
\]
satisfy $B_nf\to f$, $\partial_xB_nf\to f_x$, and $\partial_yB_nf\to f_y$,
all uniformly on $R$.
\end{lemma}
\begin{proof}
$B_nf\to f$ uniformly is the two-dimensional Bernstein--Weierstrass
theorem (same law-of-large-numbers argument as in one variable, applied
to the pair of independent Binomial$(n,x)$, Binomial$(n,y)$ counts). For
$\partial_xB_nf$, differentiating in $x$ only and applying the
one-variable identity used in Lemma~\ref{lem:c1-approx-1d} (for each
fixed $k$) gives
\[
\partial_xB_nf(x,y)=\sum_{j=0}^{n-1}\sum_{k=0}^nf_x(\xi_{j,k},\tfrac kn)
\binom{n-1}jx^j(1-x)^{n-1-j}\binom nky^k(1-y)^{n-k}
\]
for some $\xi_{j,k}\in(\tfrac jn,\tfrac{j+1}n)$, by the mean value theorem
applied to $x\mapsto f(x,\tfrac kn)$. Since $f_x$ is continuous on the
compact $R$, hence uniformly continuous, and $|\xi_{j,k}-\tfrac
jn|\le\tfrac1n$, replacing $f_x(\xi_{j,k},\tfrac kn)$ by
$f_x(\tfrac jn,\tfrac kn)$ changes the sum by at most
$\omega_{f_x}(\tfrac1n)\to0$ uniformly, using again that the basis
functions sum to $1$. The resulting main term is the tensor Bernstein
polynomial of $f_x$ (of degree $n-1$ in $x$, $n$ in $y$), which $\to f_x$
uniformly by the two-dimensional Bernstein--Weierstrass theorem applied
to the continuous function $f_x$. Hence $\partial_xB_nf\to f_x$
uniformly; symmetrically $\partial_yB_nf\to f_y$ uniformly.
\end{proof}

By an affine change of variables, Lemma~\ref{lem:c1-approx-2d} holds with
$R$ replaced by any compact rectangle; applying it on a rectangle
containing $\overline D$ on which $f$ is $C^1$ gives polynomials
$P_n\to f$, $\partial_xP_n\to\partial_xf$, uniformly on $\overline D$.
Recalling that $\oint_{\partial D}F\,dy$ means, throughout this
construction, $\int_c^d\big[F(b(y),y)-F(a(y),y)\big]\,dy$ --- an ordinary
one-dimensional integral in $y$, requiring no rectifiability of
$\partial D$ --- we get
\begin{align*}
\Big|\oint_{\partial D}P_n\,dy-\oint_{\partial D}f\,dy\Big|
&\le\int_c^d\big(|P_n-f|(b(y),y)+|P_n-f|(a(y),y)\big)\,dy\\
&\le2(d-c)\,\|P_n-f\|_\infty\longrightarrow0,
\end{align*}
using only uniform convergence $P_n\to f$ on $\overline D$; similarly
$\int_D\partial_xP_n\,dx\,dy\to\int_D\partial_xf\,dx\,dy$ from the uniform
convergence of $\partial_xP_n$ on the bounded region $D$. Hence the
identity passes to the limit:
\[
\oint_{\partial D}f(x,y)\,dy=\int_D\frac{\partial f(x,y)}{\partial x}\,dx\,dy
\]
for every $f\in C^1$ (on a neighbourhood of $\overline D$).

The extension to the three-dimensional case (and, by induction, to any
dimension) follows the same two steps --- an explicit polynomial identity
via antiderivatives, then simultaneous $C^1$-approximation via
tensor Bernstein polynomials in the extra variables --- but we do not
spell out the bookkeeping here.

\section{The Riemann--Stieltjes case}

Let $f\colon[a,b]\to\C$ be a bounded function such that
\begin{equation}\label{eq:e2}
\sum_{i=0}^{n-1}|f(t_{i+1})-f(t_i)|^2\longrightarrow0
\qquad\text{as }\delta(\pi)\to0.
\end{equation}

\begin{remark}
Condition \eqref{eq:e2} implies the continuity of $f$.
\end{remark}

Exactly as in Section~\ref{sec:real-key}, with $\alpha_i=f(t_i)$ the
telescoping identity $(m+1)I_0(\pi)+\sum_{k=1}^m\rho_k(\pi)=f(b)^{m+1}-f(a)^{m+1}$
holds, and now
\[
\rho_k(\pi)=\sum_{i=0}^{n-1}f(t_i)^{m-k}\big(f(t_{i+1})^{k-1}+\cdots+f(t_i)^{k-1}\big)\big(f(t_{i+1})-f(t_i)\big)^2,
\]
so that condition \eqref{eq:e2} controls the remainders
$\rho_k(\pi)$ (each factor in the bracket is bounded, since $f$ is
bounded), and we get the existence of
\[
\sum_{i=0}^{n-1}f(t_i)^m\big(f(t_{i+1})-f(t_i)\big)
\longrightarrow\frac{f(b)^{m+1}-f(a)^{m+1}}{m+1}
\qquad\text{as }\delta(\pi)\to0.
\]
In other words, the Riemann--Stieltjes integral exists and
\[
\int_a^bf(t)^m\,df(t)=\frac{f(b)^{m+1}-f(a)^{m+1}}{m+1}.
\]
This applies, for instance, whenever $f$ is continuous and of bounded
variation, or whenever $f$ is $\theta$-H\"older continuous with
$\theta>1/2$.

If $f,g$ are two such functions with $f(a)=g(a')$, $f(b)=g(b')$, then
clearly
\[
\int_a^bf(t)^m\,df(t)=\int_{a'}^{b'}g(t)^m\,dg(t)
=\frac{f(b)^{m+1}-f(a)^{m+1}}{m+1}.
\]
If, moreover, $f$ is continuously differentiable,
\[
\int_a^bf(t)^m\,f'(t)\,dt=\frac{f(b)^{m+1}-f(a)^{m+1}}{m+1}.
\]
By linearity, for any polynomial $P$,
\[
\int_a^bP'(f(t))\,df(t)=P(f(b))-P(f(a)),
\]
and, by a further (uniform) polynomial-approximation argument, for any
holomorphic function $\phi$ defined on a neighborhood of the (compact)
range of $f$,
\[
\int_a^b\phi'(f(t))\,df(t)=\phi(f(b))-\phi(f(a)).
\]
(A merely $C^1$ function $\phi\colon\C\to\C$, in the real sense, need not
be a uniform limit of polynomials in $z$: e.g.\ $\phi(z)=\bar z$ is
$C^\infty$ but not holomorphic anywhere, and a locally uniform limit of
polynomials is holomorphic. Holomorphy of $\phi$ is what makes the
approximation argument -- and hence this extension -- valid.)

\subsection{General integrands: $\varphi$ continuous, $g$ of bounded variation}
\label{sec:rs-bv}

Everything so far integrates $f$ (or a polynomial in $f$) against its
\emph{own} increments $df$ -- the telescoping trick relies on the
integrand being algebraically tied to $f$. We now remove that tie: let
$g\colon[a,b]\to\R$ be of bounded variation, with total variation $V(g)$,
and let $\varphi\colon[a,b]\to\C$ be continuous, unrelated to $g$. We
construct $\int_a^b\varphi(t)\,dg(t)$ directly, by the same
polynomial-first strategy as Section~\ref{sec:cont-riemann}.

\begin{lemma}\label{lem:rs-poly}
Let $P$ be a polynomial. For $\pi\in{\cal P}(a,b)$ with tags
$\tau_i\in[t_i,t_{i+1}]$, the Riemann--Stieltjes sums
\[
S(\pi):=\sum_{i=0}^{N-1}P(\tau_i)\big(g(t_{i+1})-g(t_i)\big)
\]
converge, as $\delta(\pi)\to0$, to
\[
\int_a^bP(t)\,dg(t):=P(b)g(b)-P(a)g(a)-\int_a^bP'(t)g(t)\,dt,
\]
the last integral being the ordinary Riemann integral of the continuous
function $g\cdot P'$ (Section~\ref{sec:cont-riemann}).
\end{lemma}
\begin{proof}
Since $P$ is a polynomial, $\|P'\|_\infty<\infty$ on $[a,b]$, so
$|P(\tau_i)-P(t_i)|\le\|P'\|_\infty\,\delta(\pi)$ and, since $g$ is
bounded (being of bounded variation),
\[
\Big|S(\pi)-\sum_{i=0}^{N-1}P(t_i)\big(g(t_{i+1})-g(t_i)\big)\Big|
\le\|P'\|_\infty\,\delta(\pi)\sum_{i=0}^{N-1}\big|g(t_{i+1})-g(t_i)\big|
\]
\[
\le\|P'\|_\infty\,V(g)\,\delta(\pi)\longrightarrow0.
\]
It therefore suffices to treat tags at left endpoints,
$S_0(\pi):=\sum_{i=0}^{N-1}P(t_i)\big(g(t_{i+1})-g(t_i)\big)$. By Abel
summation (summation by parts),
\[
S_0(\pi)=P(t_{N-1})g(b)-P(a)g(a)-\sum_{i=1}^{N-1}\big[P(t_i)-P(t_{i-1})\big]g(t_i).
\]
Since $|P(b)-P(t_{N-1})|\le\|P'\|_\infty\delta(\pi)$ and $g$ is bounded,
$P(t_{N-1})g(b)\to P(b)g(b)$ as $\delta(\pi)\to0$. By the mean value
theorem, $P(t_i)-P(t_{i-1})=P'(\eta_i)(t_i-t_{i-1})$ for some
$\eta_i\in(t_{i-1},t_i)$, so
\[
\sum_{i=1}^{N-1}\big[P(t_i)-P(t_{i-1})\big]g(t_i)
=\sum_{i=1}^{N-1}P'(\eta_i)g(t_i)(t_i-t_{i-1}),
\]
which is a Riemann-type sum for the continuous function $P'(t)g(t)$, with
tags $\eta_i$ within $\delta(\pi)$ of the node $t_i$ carrying the function
value $g(t_i)$: exactly as in Lemma~\ref{lem:real2}, this converges to
$\int_a^bP'(t)g(t)\,dt$ as $\delta(\pi)\to0$. Combining the three limits
gives the claim.
\end{proof}

\begin{thm}\label{thm:rs-continuous}
Let $g$ be of bounded variation on $[a,b]$ and $\varphi\colon[a,b]\to\C$
continuous. Then the Riemann--Stieltjes sums
$\sum_i\varphi(\tau_i)\big(g(t_{i+1})-g(t_i)\big)$ converge, as
$\delta(\pi)\to0$, to a limit $\int_a^b\varphi(t)\,dg(t)$, independent of
the tags $\tau_i$.
\end{thm}
\begin{proof}
Let $P_n\to\varphi$ uniformly on $[a,b]$ (Weierstrass). For any tagged
partition $\pi$,
\begin{align*}
&\Big|\sum_i\varphi(\tau_i)\big(g(t_{i+1})-g(t_i)\big)
-\sum_iP_n(\tau_i)\big(g(t_{i+1})-g(t_i)\big)\Big|\\
&\qquad\le\|\varphi-P_n\|_\infty\sum_i\big|g(t_{i+1})-g(t_i)\big|
\le\|\varphi-P_n\|_\infty\,V(g).
\end{align*}
By Lemma~\ref{lem:rs-poly}, $\sum_iP_n(\tau_i)(g(t_{i+1})-g(t_i))\to
\int_a^bP_n\,dg=:v_n$ as $\delta(\pi)\to0$, and $|v_{n+p}-v_n|\le
\|P_{n+p}-P_n\|_\infty V(g)$, so $\{v_n\}$ converges to some $v$. Exactly
as in Section~\ref{sec:cont-riemann} (the passage from Bernstein polynomials to $\phi$),
\[
\limsup_{\delta(\pi)\to0}\Big|\sum_i\varphi(\tau_i)(g(t_{i+1})-g(t_i))-v\Big|
\le\|\varphi-P_n\|_\infty\,V(g)+|v_n-v|\xrightarrow[n\to\infty]{}0,
\]
so the Riemann--Stieltjes sums converge to $v=:\int_a^b\varphi\,dg$.
\end{proof}

\begin{remark}
This is precisely the missing ingredient for boundary integrals such as
$\oint_{\partial D}\varphi\,ds$ or $\oint_{\partial D}\varphi\,dx$ over a
rectifiable boundary $\partial D$: with $g$ a coordinate function or the
arc-length of a parametrization of $\partial D$ (of bounded variation
since $\partial D$ is rectifiable), Theorem~\ref{thm:rs-continuous} gives
the integral of an arbitrary continuous $\varphi$ along the boundary,
independently of the fine structure of the parametrization.
\end{remark}

\subsection{The H\"older case: Kondurar's theorem}
\label{sec:kondurar}

Theorem~\ref{thm:rs-continuous} requires $g$ of bounded variation --
$V(g)<\infty$. This excludes genuinely rough companions: a
$\beta$-H\"older function with $\beta<1$ need not have bounded variation
(the same phenomenon already met with the H\"older paths of
Remark~\ref{rem:holder}). Kondurar~\cite{K} proved, in 1937, that bounded
variation can be traded for a matching amount of H\"older regularity on
\emph{both} sides.

\begin{lemma}[Sewing Lemma~\cite{FdlP}]\label{lem:sewing}
Let $\mu\colon\Delta_2\to\C$, $\Delta_2:=\{(s,t):a\le s\le t\le b\}$, be a
function (a \emph{germ}) such that, for some $K\ge0$ and $\gamma>0$,
\[
|\delta\mu(s,u,t)|:=|\mu(s,t)-\mu(s,u)-\mu(u,t)|\le K(t-s)^{1+\gamma}
\]
\[
\text{for all }a\le s\le u\le t\le b.
\]
Then the Riemann-type sums $\sum_i\mu(t_i,t_{i+1})$ converge, as
$\delta(\pi)\to0$, to a limit $I(a,b)$, independent of the partition,
with the quantitative estimate
$|\mu(s,t)-I(s,t)|\le C(\gamma)K(t-s)^{1+\gamma}$
for every $s\le t$, where $I(s,t):=\lim_{\delta(\pi)\to0}\sum_i\mu(t_i,t_{i+1})$
over partitions of $[s,t]$.
\end{lemma}

We do not reprove Lemma~\ref{lem:sewing} here; we use it as a black box,
verifying its hypothesis for the specific germ arising from
Riemann--Stieltjes sums.

\begin{cor}[Young--Kondurar theorem, via the Sewing Lemma]\label{thm:kondurar}
Let $f,\varphi\colon[a,b]\to\C$ satisfy
\[
|f(u)-f(v)|\le M|u-v|^\alpha,\qquad |\varphi(u)-\varphi(v)|\le N|u-v|^\beta,
\qquad u,v\in[a,b],
\]
with $\alpha+\beta=1+\gamma$, $\gamma>0$. Then the Riemann--Stieltjes sums
\[
S(\pi):=\sum_{i=0}^{N-1}f(\tau_i)\big(\varphi(t_{i+1})-\varphi(t_i)\big)
\]
converge, as $\delta(\pi)\to0$, to a limit
$\int_a^bf(t)\,d\varphi(t)$, independent of the tags $\tau_i$.
\end{cor}
\begin{proof}
Apply Lemma~\ref{lem:sewing} to the germ $\mu(s,t):=f(s)(\varphi(t)-\varphi(s))$,
whose Riemann-type sums $\sum_i\mu(t_i,t_{i+1})=\sum_if(t_i)(\varphi(t_{i+1})-\varphi(t_i))$
are exactly the left-tagged case of $S(\pi)$ (the general tag $\tau_i$
in place of $t_i$ is handled exactly as in the proof of
Theorem~\ref{thm:rs-continuous}, via $|f(\tau_i)-f(t_i)|\le M\delta(\pi)^\alpha$,
so it suffices to treat $\tau_i=t_i$). For $a\le s\le u\le t\le b$, with
$\ell:=t-s$,
\[
\delta\mu(s,u,t)=\mu(s,t)-\mu(s,u)-\mu(u,t)
\]
\[
=f(s)(\varphi(t)-\varphi(u))-f(u)(\varphi(t)-\varphi(u))
\]
\[
=[f(s)-f(u)](\varphi(t)-\varphi(u)),
\]
so, using the H\"older conditions and $|u-s|,|t-u|\le\ell$,
\[
|\delta\mu(s,u,t)|\le M|s-u|^\alpha\cdot N|t-u|^\beta\le MN\,\ell^{\alpha+\beta}
=MN\,\ell^{1+\gamma}.
\]
This is exactly the hypothesis of Lemma~\ref{lem:sewing}, with $K=MN$;
the conclusion follows.
\end{proof}

\begin{remark}[Kondurar's original elementary argument]\label{rem:kondurar-elementary}
Kondurar's 1937 proof~\cite{K} predates the Sewing Lemma by about seventy
years and does not invoke it; the estimate above -- quasi-additivity of
$\mu(s,t)=f(s)(\varphi(t)-\varphi(s))$ under a single-point split,
$|\delta\mu(s,u,t)|\le MN\ell^{1+\gamma}$ -- is exactly what his argument
establishes directly, by dyadic bisection of a single interval $[u,v]$
of length $\ell$: writing $\Sigma_k$ for the dyadic Riemann--Stieltjes
sum of $f\,d\varphi$ at level $k$ (so $\Sigma_0=f(\xi)(\varphi(v)-\varphi(u))$
for a tag $\xi\in[u,v]$), comparing consecutive levels gives
\[
|\Sigma_{k-1}-\Sigma_k|\lesssim MN\Big(\frac1{2^{k-1}}\Big)^{\gamma}\ell
\]
(only the $\sim2^{k-1}$ segments new at level $k$, each of length
$\sim\ell/2^{k-1}$, contribute, each by at most
$M(\ell/2^{k-1})^\alpha\cdot N(\ell/2^{k-1})^\beta$), and summing the
resulting geometric series over $k$ shows $(\Sigma_k)_k$ is Cauchy, hence
convergent -- this is the same quasi-additivity estimate, derived by an
explicit dyadic ``median segment/lateral segments'' bookkeeping rather
than cited from an abstract lemma. What Kondurar's elementary argument
does \emph{not}, by itself, immediately give is convergence for
\emph{arbitrary} (non-dyadic) partitions $\pi$: that step -- inserting
the points separating an arbitrary $\pi$ from a common refinement,
organized by scale so that the same geometric series controls the total
error -- is precisely the general bookkeeping that Lemma~\ref{lem:sewing}
performs once and for all, for any germ satisfying the quasi-additivity
hypothesis; Kondurar's own article carries out the equivalent
bookkeeping directly, without isolating it as a reusable abstract
statement.
\end{remark}

\begin{remark}
The exponent condition $\alpha+\beta>1$ is exactly Young's condition
\cite{Y}: Kondurar's article, though published in 1937, was received by
the editors in August 1936, the same year as Young's paper (Acta
Mathematica, December 1936) -- the two results are independent and
essentially contemporaneous, not one derivative of the other. Corollary
\ref{thm:kondurar} is what makes good, at last, the claim of
Remark~\ref{rem:holder}, that the complex case's H\"older extension
``recovers the setting of Young's integral'': that remark asserted the
fact, and it is Corollary~\ref{thm:kondurar} (specialized to $\varphi=\gamma$
a H\"older path and $f$ H\"older with matching exponent) that actually
proves it. In this sense Kondurar's 1937 argument is, in effect, a proof
of a special case of the Sewing Lemma -- for the germ
$\mu(a,b)=f(a)(\varphi(b)-\varphi(a))$ -- some seventy years before the
general statement was isolated, a historical point we regard as
interesting but state cautiously: we have verified the coincidence of
the underlying quasi-additivity estimate explicitly above
(Remark~\ref{rem:kondurar-elementary}), rather than merely asserting
the connection.
\end{remark}

\part{Stochastic extension}

\section{The Wiener process case and semimartingales}\label{sec:realWiener}

We now apply the same idea to integrate polynomials of a Wiener process
$W$ with respect to $W$ itself, recovering a particular case of the It\^o
formula.

Given a partition $\pi\colon a=t_0<t_1<\cdots<t_n=b$, with $0\le a<b$ (as
is implicit throughout below in $\E[W_t^{2j}]=c_jt^j$, $j\ge1$), and a
fixed integer
$m$, set, as before,
\[
I_k(\pi)=\ds\sum_{i=0}^{n-1}W_{t_i}^{m-k}W_{t_{i+1}}^k(W_{t_{i+1}}-W_{t_i}),
\qquad k=0,\ldots,m,
\]
whose sum over $k$ is $W_b^{m+1}-W_a^{m+1}$, independent of $\pi$. Up to
this point nothing differs from the real case of
Section~\ref{sec:real-key}. Writing again
$I_k(\pi)=I_0(\pi)+\rho_k(\pi)$,
\begin{equation}\label{eq:e3}
W_b^{m+1}-W_a^{m+1}=(m+1)I_0(\pi)+\sum_{k=1}^m\rho_k(\pi),
\end{equation}
with
\[
\rho_k(\pi)=\sum_{i=0}^{n-1}W_{t_i}^{m-k}
\big(W_{t_{i+1}}^{k-1}+\cdots+W_{t_i}^{k-1}\big)(W_{t_{i+1}}-W_{t_i})^2
=\sum_{j=0}^{k-1}\rho_{k,j}(\pi),
\]
\begin{align*}
\rho_{k,j}(\pi)&=\sum_{i=0}^{n-1}W_{t_i}^{m-j-1}W_{t_{i+1}}^j(W_{t_{i+1}}-W_{t_i})^2\\
&=\sum_{\nu=0}^j\bin j\nu\sum_{i=0}^{n-1}W_{t_i}^{m-\nu-1}(W_{t_{i+1}}-W_{t_i})^{\nu+2}
=:\sum_{\nu=0}^j\bin j\nu\,\rho_{k,j,\nu}(\pi).
\end{align*}

\begin{lemma}\label{lem:l1}
The sums $\sum_{i=0}^{n-1}W_{t_i}^k(W_{t_{i+1}}-W_{t_i})^\ell$ converge, for
$\ell\ge3$, to $0$ in $L^2(\Omega)$; for $\ell=2$,
\[
\sum_{i=0}^{n-1}W_{t_i}^k(W_{t_{i+1}}-W_{t_i})^2\longrightarrow\int_a^bW_t^k\,dt
\qquad\text{in }L^2(\Omega).
\]
\end{lemma}
\begin{proof}
Write the sum in \eqref{e4} below as $\sum_i W_{t_i}^k\xi_i^\ell(t_{i+1}-t_i)^{\ell/2}$,
where
\begin{equation}\label{e4}
\xi_i=\frac{W_{t_{i+1}}-W_{t_i}}{\sqrt{t_{i+1}-t_i}}
\end{equation}
are i.i.d.\ $N(0,1)$, each independent of $W_{t_i}$.

For $\ell=2$,
\[
\sum_{i=0}^{n-1}W_{t_i}^k\xi_i^2(t_{i+1}-t_i)=
\sum_{i=0}^{n-1}W_{t_i}^k(\xi_i^2-1)(t_{i+1}-t_i)+\sum_{i=0}^{n-1}W_{t_i}^k(t_{i+1}-t_i).
\]
With $c_k=(2k-1)!!=\E(\xi^{2k})$ for $\xi\sim N(0,1)$, the first term satisfies
\[
\E\Big(\sum_{i=0}^{n-1}W_{t_i}^k(\xi_i^2-1)(t_{i+1}-t_i)\Big)^2
=2c_k\sum_{i=0}^{n-1}t_i^k(t_{i+1}-t_i)^2
\]
\[
\le2c_k\,\delta(\pi)\sum_{i=0}^{n-1}t_i^k(t_{i+1}-t_i)\longrightarrow0,
\]
so it converges to $0$ in $L^2(\Omega)$. The second term is a Riemann sum,
with deterministic weights $t_{i+1}-t_i$, of the continuous (hence
pathwise bounded on $[a,b]$) function $t\mapsto W_t^k(\omega)$, so it
converges a.e.\ to $\int_a^bW_t^k\,dt$. To upgrade this to $L^2(\Omega)$
convergence it suffices to bound its fourth moment uniformly in $\pi$: by
Jensen's inequality, applied to the convex function $x\mapsto x^4$ with
the weights $(t_{i+1}-t_i)/(b-a)$ (which sum to $1$),
\begin{align*}
\Big(\sum_{i=0}^{n-1}W_{t_i}^k(t_{i+1}-t_i)\Big)^4
&=(b-a)^4\Big(\sum_{i=0}^{n-1}\frac{t_{i+1}-t_i}{b-a}\,W_{t_i}^k\Big)^4\\
&\le(b-a)^3\sum_{i=0}^{n-1}(t_{i+1}-t_i)\,W_{t_i}^{4k}.
\end{align*}
Taking expectations and using $\E[W_t^{4k}]=c_{2k}t^{2k}\le c_{2k}b^{2k}$
for $t\in[a,b]$,
\[
\E\Big(\sum_{i=0}^{n-1}W_{t_i}^k(t_{i+1}-t_i)\Big)^4
\le(b-a)^3\sum_{i=0}^{n-1}(t_{i+1}-t_i)\,c_{2k}b^{2k}=(b-a)^4c_{2k}\,b^{2k},
\]
a bound independent of $\pi$. This uniform $L^4$ bound gives uniform
integrability of the squares of the (a.e.-convergent) second term, hence
convergence in $L^2(\Omega)$.

For $\ell\ge3$ odd,
\[
\E\Big(\sum_{i=0}^{n-1}W_{t_i}^k\xi_i^\ell(t_{i+1}-t_i)^{\ell/2}\Big)^2
=c_kc_\ell\sum_{i=0}^{n-1}t_i^k(t_{i+1}-t_i)^\ell
\]
\[
\le c_kc_\ell\,\delta(\pi)^{\ell-1}\sum_{i=0}^{n-1}t_i^k(t_{i+1}-t_i)\longrightarrow0.
\]
For $\ell\ge3$ even, split as for $\ell=2$,
\begin{align*}
\sum_{i=0}^{n-1}W_{t_i}^k\xi_i^\ell(t_{i+1}-t_i)^{\ell/2}
&=\sum_{i=0}^{n-1}W_{t_i}^k(\xi_i^\ell-c_{\ell/2})(t_{i+1}-t_i)^{\ell/2}\\
&\qquad+c_{\ell/2}\sum_{i=0}^{n-1}W_{t_i}^k(t_{i+1}-t_i)^{\ell/2};
\end{align*}
the first term $\to0$ in $L^2(\Omega)$ as before. For the second term,
write $(t_{i+1}-t_i)^{\ell/2}=(t_{i+1}-t_i)^{\ell/2-1}(t_{i+1}-t_i)$ and
bound $(t_{i+1}-t_i)^{\ell/2-1}\le\delta(\pi)^{\ell/2-1}$; by the same
Jensen argument used above, now with the convex function $x\mapsto x^2$,
\[
\E\Big(\sum_{i=0}^{n-1}W_{t_i}^k(t_{i+1}-t_i)\Big)^2
\le(b-a)\sum_{i=0}^{n-1}(t_{i+1}-t_i)\,\E[W_{t_i}^{2k}]
\le(b-a)^2c_k\,b^k=:C_k,
\]
a bound independent of $\pi$, so
\[
\E\Big(\sum_{i=0}^{n-1}W_{t_i}^k(t_{i+1}-t_i)^{\ell/2}\Big)^2
\le\delta(\pi)^{\ell-2}\,C_k\longrightarrow0,
\]
since $\ell>2$: the second term tends to $0$ in $L^2(\Omega)$. It also
tends to $0$ a.e.\ directly, along the net $\delta(\pi)\to0$: for
$\omega$ in the a.s.\ event that $t\mapsto W_t(\omega)$ is continuous,
$|W_t(\omega)|^k\le M(\omega):=\sup_{t\in[a,b]}|W_t(\omega)|^k<\infty$,
so $\big|\sum_iW_{t_i}(\omega)^k(t_{i+1}-t_i)^{\ell/2}\big|
\le\delta(\pi)^{\ell/2-1}M(\omega)(b-a)\to0$ as $\delta(\pi)\to0$, for
this fixed $\omega$ and \emph{every} partition of small enough mesh ---
no subsequence needed, unlike the situation of
Proposition~\ref{rem:net-vs-seq}, because here the bound is a deterministic
function of $\delta(\pi)$ alone, uniform over all partitions of a given
mesh, rather than a probabilistic statement requiring a topology to test
against nets. This a.e.\ argument applies only to this second,
deterministic-mesh-decay piece; the centred fluctuation term (present for
every $\ell\ge3$: the whole term when $\ell$ is odd, and the first
summand of the splitting above when $\ell$ is even) is bounded above only
in $L^2(\Omega)$, via the second-moment estimates used throughout this
proof, with no accompanying a.e.\ argument. Combining the two pieces by
the triangle inequality therefore yields convergence in $L^2(\Omega)$
for the full sum, for every $\ell\ge3$, but not a pathwise a.e.\
statement for the net $\delta(\pi)\to0$ --- which is why
Lemma~\ref{lem:l1} is stated in $L^2(\Omega)$ only.
\end{proof}

By Lemma~\ref{lem:l1}, $\rho_{k,j,0}(\pi)\to\int_a^bW_t^{m-1}\,dt$ in
$L^2(\Omega)$, while $\rho_{k,j,\nu}(\pi)\to0$ in $L^2(\Omega)$
for $\nu\ge1$. Hence
\[
\rho_k(\pi)\longrightarrow k\int_a^bW_t^{m-1}\,dt
\quad\text{in }L^2(\Omega),
\]
\[
\sum_{k=1}^m\rho_k(\pi)\longrightarrow\frac{m(m+1)}2\int_a^bW_t^{m-1}\,dt
\quad\text{in }L^2(\Omega),
\]
as $\delta(\pi)\to0$. From \eqref{eq:e3}\footnote{This time the various
$I_k$ converge, in $L^2(\Omega)$, to \emph{different} limits.},
\[
I_k(\pi)\longrightarrow\frac{W_b^{m+1}-W_a^{m+1}}{m+1}-\Big(\frac m2-k\Big)\int_a^bW_t^{m-1}\,dt
\qquad\text{in }L^2(\Omega),
\]
and in particular, for $k=0$,
\[
I_0(\pi)\longrightarrow\frac{W_b^{m+1}-W_a^{m+1}}{m+1}-\frac m2\int_a^bW_t^{m-1}\,dt
\qquad\text{in }L^2(\Omega).
\]
We take this $L^2(\Omega)$-limit of $I_0(\pi)$ as the \emph{definition}
of $\int_a^bW(t)^m\,dW_t$ for this integrand, rather than presupposing
the classical It\^o integral and merely renaming a limit that is proved
to coincide with it: we write
\[
W_b^{m+1}-W_a^{m+1}=(m+1)\int_a^bW(t)^m\,dW_t+\frac{m(m+1)}2\int_a^bW_t^{m-1}\,dt,
\]
and, by linearity, for any polynomial $P$,
\[
P(W_b)-P(W_a)=\int_a^bP'(W(t))\,dW(t)+\frac12\int_a^bP''(W_t)\,dt,
\]
which is the classical It\^o formula.

\begin{proposition}[Identification with the classical It\^o integral]\label{prop:ito-identification}
The definition just adopted agrees with the classical It\^o integral,
built by $L^2$-approximation of predictable integrands by elementary
(simple, piecewise-constant, adapted) processes: the left-point sums
$I_0(\pi)$ are themselves elementary stochastic integrals of simple
predictable processes converging to $W(t)^m$ in the norm used for the
It\^o isometry, so their $L^2(\Omega)$-limit, which is what
$\int_a^bW(t)^m\,dW_t$ was just defined to be, coincides with the
classical It\^o integral of $W(t)^m$, not merely in name but as the
$L^2(\Omega)$-limit of one and the same approximating family.
\end{proposition}
\begin{proof}
For a fixed
partition $\pi$, the left-point sum
$I_0(\pi)=\sum_iW_{t_i}^m(W_{t_{i+1}}-W_{t_i})$ is, by definition, the
elementary stochastic integral $\int_a^b E_\pi(t)\,dW_t$ of the simple
predictable process $E_\pi(t):=W_{t_i}^m$ for $t\in[t_i,t_{i+1})$ (simple
because it is constant on each $[t_i,t_{i+1})$, predictable because
$W_{t_i}^m$ is $\mathcal F_{t_i}$-measurable); and the classical It\^o
integral of $W(t)^m$ is, by construction, the $L^2(\Omega)$-limit of
exactly such elementary integrals, along any sequence of simple
predictable processes converging to $W(t)^m$ in the $L^2([a,b]\times
\Omega)$ norm used for the isometry --- which the left-point processes
$E_\pi$ do, as $\delta(\pi)\to0$, since $W(t)^m$ has continuous (hence
$L^2$-bounded on the compact $[a,b]$) paths; explicitly, the relevant
integrands are square-integrable in the sense required by the isometry,
$\E\int_a^b|W_t|^{2m}\,dt<\infty$, since $t\mapsto\E|W_t|^{2m}=c_mt^m$
is continuous, hence bounded, on the compact $[a,b]$. So the limit constructed
above is not merely notated as the It\^o integral: it is literally the
$L^2$-limit of the same elementary approximations used in the standard
construction of that integral, for this particular choice of
approximating sequence.
\end{proof}

\begin{remark}
More generally, one may define the limit of $I_k(\pi)$ as the
$k$-stochastic integral $\int_a^bW_t^m\,d^{(k)}W_t$ --- the notation
retaining the dependence on the fixed polynomial degree $m$ appearing in
$I_k(\pi)$, since the same construction with a different $m$ gives a
different integral ---
\[
\int_a^bW_t^m\,d^{(k)}W_t:=\lim_{\delta(\pi)\to0}I_k(\pi)
=\frac{W_b^{m+1}-W_a^{m+1}}{m+1}-\Big(\frac m2-k\Big)\int_a^bW_t^{m-1}\,dt,
\]
which leads to a $k$-It\^o formula
\[
P(W_b)-P(W_a)=\int_a^bP'(W(t))\,dW(t)+\Big(\frac12-\frac km\Big)\int_a^bP''(W_t)\,dt.
\]
In particular, averaging the two extreme cases $k=0$ and $k=m$ cancels
the correction term entirely, since their coefficients $\tfrac12-\tfrac
km$ are $\tfrac12$ and $-\tfrac12$: this average is precisely the
Stratonovich integral, characterized by satisfying the ordinary chain
rule with no second-derivative correction. We do not claim, and have not
verified, that other linear combinations of the $I_k(\pi)$ define
stochastic integrals with any comparable interpretation; the averaging
identity above is the only such combination this note establishes.
\end{remark}

\begin{remark}[Forward, backward, and F\"ollmer's own choice]
The case $k=0$ (left-point sums, $I_0(\pi)$) is exactly the tag
convention F\"ollmer uses in his own pathwise construction of the It\^o
integral~\cite{F}: writing his compensated sum, in the notation of
\cite{ContPerkowski19}, as $\sum\nabla f(S(t_j))\cdot(S(t_{j+1})-S(t_j))$,
the integrand is always evaluated at the \emph{left} endpoint $t_j$ --
the forward, It\^o convention, our $k=0$. Had F\"ollmer instead tagged
at the right endpoint $t_{j+1}$ -- our $k=m$ -- the same telescoping
argument would define the \emph{backward} stochastic integral in place
of the forward one, with the correction term's sign reversed
($+\frac12$ instead of $-\frac12$, matching the coefficient $\tfrac12-
\tfrac km$ above at $k=m$). Stratonovich's integral, with no correction
at all, sits exactly halfway between the two -- which is one more way
of stating the averaging identity just proved.
\end{remark}

\subsection{Extension to semimartingales}

Before stating the next result, we make explicit the exact mode of
convergence used throughout this subsection, since it is easy to leave
implicit and the manuscript's convergence statements are given in the
net (Moore--Smith) sense of Proposition~\ref{rem:net-vs-seq}, not merely
along a sequence. Concretely, ``$Y_\pi\to Y$ in probability as
$\delta(\pi)\to0$'', for a family of random variables $(Y_\pi)_{\pi\in
D}$ indexed by the directed set $D$ of partitions of $[a,b]$ (ordered
by $\delta(\pi')\le\delta(\pi)$, as in Section~\ref{sec:realWiener}),
means
\[
\forall\,\eps>0\ \forall\,\rho>0\ \ \exists\,\eta>0\ \ \text{such that}\ \
\delta(\pi)<\eta\ \Longrightarrow\ P\big(|Y_\pi-Y|>\eps\big)<\rho,
\]
i.e.\ the $\eta$ in the definition of net convergence must work
\emph{simultaneously for every partition} of mesh below it -- not merely
along one sequence of partitions with mesh $\to0$. By
Proposition~\ref{rem:net-vs-seq} below, this is equivalent to the same
statement holding along \emph{every} sequence of partitions with mesh
$\to0$, which is the form in which convergence in probability is usually
stated in the stochastic-calculus literature; we use the net formulation
here because it is what the hypothesis of Lemma~\ref{lem:l2} below is
naturally stated in (a limit taken over the fixed directed set
$\mathcal P$ of partitions, uniformly in $\pi$, rather than along a
sequence chosen in advance), and Proposition~\ref{rem:net-vs-seq},
proved later in this section, is precisely what licenses passing freely
between the two formulations wherever needed.

\begin{lemma}\label{lem:l2}
Let $X_t$, $t\in[a,b]$, have continuous paths (a.s.), and let $A_t$,
$t\in[a,b]$, be a continuous, non-decreasing process such that, for
every $s\le t$ in $[a,b]$ and every partition $\pi$ \emph{containing}
$s$ and $t$ as partition points (i.e.\ $s,t\in\{t_0,\ldots,t_n\}$, so
that the sub-sum below is unambiguous, with no boundary fragment left
over at either end),
\[
\sum_{i\,:\,s\le t_i<t_{i+1}\le t}(X_{t_{i+1}}-X_{t_i})^2
\longrightarrow A_t-A_s\qquad\text{in probability, as }\delta(\pi)\to0
\]
along such $s,t$-containing partitions (taking $s=a$, $t=b$ gives, in
particular, $\sum_{i=0}^{n-1}(X_{t_{i+1}}-X_{t_i})^2\to A_b-A_a$, and
every partition trivially contains $a,b$ as its endpoints). Then, for
$\ell\ge3$,
\[
\sum_{i=0}^{n-1}X_{t_i}^k(X_{t_{i+1}}-X_{t_i})^\ell\longrightarrow0
\qquad\text{in probability},
\]
while for $\ell=2$,
\[
\sum_{i=0}^{n-1}X_{t_i}^k(X_{t_{i+1}}-X_{t_i})^2\longrightarrow\int_a^bX_t^k\,dA_t
\qquad\text{in probability}.
\]
\end{lemma}

\begin{proof}
Throughout, write $\Delta X_i:=X_{t_{i+1}}-X_{t_i}$ for the increments of a
generic partition $\pi\colon a=t_0<\cdots<t_n=b$. We do \emph{not} use a
semimartingale decomposition $X=Y+M$, nor any martingale property of $A$:
the conclusion follows directly from the two stated hypotheses ---
convergence of $\sum(\Delta X_i)^2$ to $A_t-A_s$ on every sub-interval
$[s,t]$, along partitions containing $s,t$ as grid points, and continuity
of the path $t\mapsto X_t$ --- together with the
monotonicity of $A$, used below. This is the same elementary
Riemann--Stieltjes convergence argument used, in the deterministic
setting, throughout this note, now carried out ``in probability'' instead
of pointwise, and it requires nothing from the general theory of
semimartingales.

\medskip\noindent\textit{Step 1: the case $\ell=2$ for a simple weight.}
Let $\phi$ be a (non-random, or more generally $\mathcal F_a$-measurable)
step function on $[a,b]$, $\phi(t)=c_j$ for $t\in[s_j,s_{j+1})$, for a
fixed partition $a=s_0<s_1<\cdots<s_N=b$, and set
$M_\phi:=\max_j|c_j|<\infty$. For $\pi$ \emph{containing} $\{s_j\}$
(i.e.\ $\{s_j\}\subset\{t_0,\ldots,t_n\}$),
\begin{align*}
\sum_i\phi(t_i)(\Delta X_i)^2&=\sum_{j=0}^{N-1}c_j\!\!
\sum_{i\,:\,t_i\in[s_j,s_{j+1})}\!\!(\Delta X_i)^2\\
&\ \xrightarrow[\delta(\pi)\to0]{P}\ \sum_{j=0}^{N-1}c_j(A_{s_{j+1}}-A_{s_j})
=:\int_a^b\phi(t)\,dA_t,
\end{align*}
using the hypothesis on each of the finitely many fixed sub-intervals
$[s_j,s_{j+1}]$ and the fact that a finite sum of terms, each converging
in probability, converges in probability. This proves the $\ell=2$
statement for a step-function weight in place of $X_{t_i}^k$, for
partitions containing $\{s_j\}$, directly from the hypothesis.

\smallskip\noindent\emph{Extension to arbitrary $\pi$.} The hypothesis
of the Lemma, and hence the display above, is stated only for partitions
containing the fixed points $\{s_j\}$; we now remove this restriction,
for this same fixed $\phi$, since it is needed for what follows. Given
an arbitrary partition $\pi\colon a=t_0<\cdots<t_n=b$ (not necessarily
containing $\{s_j\}$), let $\widehat\pi:=\pi\cup\{s_j\}$ be the partition
obtained by inserting the (at most $N-1$ interior) points of $\{s_j\}$
not already in $\pi$; then $\widehat\pi\supset\{s_j\}$, and inserting
points can only shrink cells, so $\delta(\widehat\pi)\le\delta(\pi)$.
At most $N-1$ cells of $\pi$ are split by this insertion (one per
interior point $s_j$); every other cell of $\pi$ is a cell of
$\widehat\pi$ too, contributing identically to both sums below.
For a split cell $[t_i,t_{i+1}]\ni s_j$ (in its interior), the
$\pi$-sum contributes $\phi(t_i)(X_{t_{i+1}}-X_{t_i})^2$, while the
$\widehat\pi$-sum contributes
$\phi(t_i)(X_{s_j}-X_{t_i})^2+\phi(s_j)(X_{t_{i+1}}-X_{s_j})^2$; both
are, in modulus, at most $M_\phi\,\delta(\pi)^2$ (each factor
$|X_{t_{i+1}}-X_{t_i}|,|X_{s_j}-X_{t_i}|,|X_{t_{i+1}}-X_{s_j}|$ is at
most $\delta(\pi)$, being a sub-increment of a cell of $\pi$ of length
$\le\delta(\pi)$), so their difference is at most $3M_\phi\,\delta(\pi)^2$.
Summing over the at most $N-1$ affected cells,
\[
\Big|\sum_i\phi(t_i)(\Delta X_i)^2-\sum_i\phi(\hat t_i)(\Delta\widehat
X_i)^2\Big|\le3(N-1)M_\phi\,\delta(\pi)^2\longrightarrow0
\]
\emph{deterministically}, as $\delta(\pi)\to0$ -- with $N$ fixed (it
depends only on $\phi$, chosen before $\pi$), this bound has nothing
random or probabilistic left in it, unlike every other estimate in this
proof. Since $\widehat\pi\supset\{s_j\}$ and
$\delta(\widehat\pi)\le\delta(\pi)\to0$, the first display of this Step
applies to $\widehat\pi$, giving
$\sum_i\phi(\hat t_i)(\Delta\widehat X_i)^2\to\int_a^b\phi\,dA_t$ in
probability as $\delta(\pi)\to0$; combined with the deterministic bound
above,
\[
\sum_i\phi(t_i)(\Delta X_i)^2\ \xrightarrow[\delta(\pi)\to0]{P}\
\int_a^b\phi(t)\,dA_t
\]
for \emph{every} partition $\pi$, not only those containing $\{s_j\}$.
This is the form of Step~1 used in Step~2 below.

\medskip\noindent\textit{Step 2: the case $\ell=2$ for $X_t^k$, by
approximation.} Since $X$ has continuous paths a.s.\ and $[a,b]$ is
compact, the (random) modulus of continuity
\[
\omega_{X^k}(\eta):=\sup_{\substack{s,t\in[a,b]\\|t-s|\le\eta}}|X_t^k-X_s^k|
\]
satisfies $\omega_{X^k}(\eta)\to0$ a.s.\ as $\eta\to0^+$ (pathwise
uniform continuity on the compact $[a,b]$), hence also in probability.
So, given $\eps>0$ and $\rho>0$, there is a \emph{deterministic}
$\eta=\eta(\eps,\rho)>0$ --- not depending on $\omega$ --- such that
\[
P\big(\omega_{X^k}(\eta)>\eps\big)<\rho.
\]
Fix, once and for all, a deterministic partition $a=s_0<\cdots<s_N=b$ of
mesh $<\eta$ (independent of $\omega$), and define the step function
$\phi(t):=X_{s_j}^k$ for $t\in[s_j,s_{j+1})$; $\phi$ is random through
its values $X_{s_j}^k$, but its underlying partition $\{s_j\}$ is fixed,
so Step~1 (in the form just established for arbitrary partitions)
applies to it directly, without restricting $\pi$. On the event
$\{\omega_{X^k}(\eta)\le\eps\}$, of probability $\ge1-\rho$,
$\sup_{t\in[a,b]}|X_t^k-\phi(t)|\le\eps$ pathwise. For \emph{every}
partition $\pi$, on this event,
\begin{align*}
\Big|\sum_iX_{t_i}^k(\Delta X_i)^2-\sum_i\phi(t_i)(\Delta X_i)^2\Big|
&\le\Big(\sup_t|X_t^k-\phi(t)|\Big)\sum_i(\Delta X_i)^2\\
&\le\eps\sum_i(\Delta X_i)^2,
\end{align*}
and likewise $\big|\int_a^bX_t^k\,dA_t-\int_a^b\phi(t)\,dA_t\big|
\le\eps(A_b-A_a)$, since $A$ is non-decreasing. Since $\sum_i(\Delta
X_i)^2\to A_b-A_a$ in probability (the hypothesis, $t=b$), it is bounded
in probability: given $\rho>0$ there is $M$ with
$P\big(\sum_i(\Delta X_i)^2>M\big)<\rho$ for all $\pi$ with $\delta(\pi)$
small enough. On the intersection of this event with
$\{\omega_{X^k}(\eta)\le\eps\}$ --- of probability $\ge1-2\rho$, for the
deterministic $\eta=\eta(\eps,\rho)$ fixed above --- both error terms are
at most $\eps M$ and $\eps(A_b-A_a)$ respectively; combined with Step~1
applied to this fixed, deterministic $\phi$, a standard ``$3\eps$''
argument gives
\[
\sum_iX_{t_i}^k(\Delta X_i)^2\ \xrightarrow[\delta(\pi)\to0]{P}\ \int_a^bX_t^k\,dA_t,
\]
which is the $\ell=2$ statement, established directly from the hypothesis
and the continuity of $X$, without any martingale argument. Working
throughout with a fixed, deterministic partition $\{s_j\}$ --- rather
than a partition constructed after fixing $\omega$ --- keeps the
approximation argument and the ``in probability'' statement cleanly
separated.

\medskip\noindent\textit{Step 3: the case $\ell\ge3$.} By Step~2 (with
$|X_t|^k$ in place of $X_t^k$), $\sum_i|X_{t_i}|^k(\Delta X_i)^2\to
\int_a^b|X_t|^k\,dA_t$ in probability, so this quantity is bounded in
probability (tight) as $\delta(\pi)\to0$. Since $X$ is (a.s.) uniformly
continuous on $[a,b]$, $\eta(\pi):=\max_i|\Delta X_i|\to0$ a.s.; then, for
$\ell\ge3$,
\[
\Big|\sum_iX_{t_i}^k(\Delta X_i)^\ell\Big|
\le\eta(\pi)^{\ell-2}\sum_i|X_{t_i}|^k(\Delta X_i)^2,
\]
a product of a term $\to0$ a.s.\ (hence in probability) and a tight term,
hence itself $\to0$ in probability. (As in Lemma~\ref{lem:l1}, this
argument gives convergence in probability for the net $\delta(\pi)\to0$;
an almost-sure statement is \emph{not} established by this route and
would need either a fixed sequence of partitions with
$\sum_n\delta(\pi_n)<\infty$ together with a Borel--Cantelli argument, or
the $L^2$-theory for the quadratic variation of a bounded martingale
together with passage to a subsequence.)
\end{proof}

\begin{remark}
This proof is more elementary than one built on the semimartingale
decomposition $X=Y+M$: it never separates a bounded-variation part from a
martingale part, never invokes the defining martingale property of
$\langle M\rangle$ (hence no Doob--Meyer decomposition), and uses no
conditional expectations at all. The only inputs are (i) the hypothesis
itself, applied on finitely many fixed sub-intervals (Step~1), and (ii)
the classical $\eps$-approximation of a continuous function by step
functions (Step~2) — exactly the mechanism already used, in the
deterministic setting, to establish Riemann--Stieltjes convergence
elsewhere in this note, now run ``in probability''. The semimartingale
decomposition $X=Y+M$ introduced below is \emph{not} needed for
Lemma~\ref{lem:l2} itself: it is needed only afterwards, to identify the
process $A_t$ appearing in the Lemma's hypothesis with the quadratic
variation $\langle M\rangle_t$ of the martingale part of a concrete
semimartingale $X$.
\end{remark}

\begin{proposition}[Nets versus sequences]\label{rem:net-vs-seq}
Let $D$ be the set of partitions of $[a,b]$, preordered by
$\pi\preceq\pi'$ iff $\delta(\pi)\ge\delta(\pi')$ (a directed preorder,
not a partial order: two distinct partitions of equal mesh are
equivalent for it, which causes no difficulty, since Moore--Smith
convergence only requires $D$ to be a directed preorder). Let $X$ be a
Hausdorff space and $(S_\pi)_{\pi\in D}$ a net in $X$. Then $S_\pi\to x$
as $\delta(\pi)\to0$ (net convergence) if and only if $S_{\pi_n}\to x$
for every sequence $(\pi_n)\subset D$ with $\delta(\pi_n)\to0$
(sequential convergence along arbitrary vanishing-mesh sequences, with
no relation required between successive $\pi_n$).
\end{proposition}
\begin{proof}
Recall $S_\pi\to x$ means: for every open $G\ni x$ there is $\pi_G\in D$
such that $S_\pi\in G$ for all $\pi\succeq\pi_G$. Suppose the net does
\emph{not} converge to $x$: there is an open $G_0\ni x$ such that for
every $\gamma\in D$ there exists $\pi_\gamma\succeq\gamma$ with
$S_{\pi_\gamma}\notin G_0$. Fix any sequence $\gamma_n\in D$ with
$\delta(\gamma_n)<1/n$, and apply this to each $\gamma=\gamma_n$: there
is $\pi_n:=\pi_{\gamma_n}\succeq\gamma_n$ with $S_{\pi_n}\notin G_0$.
Since the order is defined purely by mesh comparison,
$\pi_n\succeq\gamma_n$ forces $\delta(\pi_n)\le\delta(\gamma_n)<1/n$; so
$\{\pi_n\}$ is a genuine sequence of partitions with $\delta(\pi_n)\to0$
along which $S_{\pi_n}$ stays outside the fixed neighbourhood $G_0$ for
every $n$, hence does \emph{not} converge to $x$. Contrapositively: if
$S_{\pi_n}\to x$ along every sequence with mesh $\to0$, the net
converges. The converse is immediate, since any such sequence eventually
lies in every neighbourhood used to define the net limit.
\end{proof}

\begin{remark}
Throughout this note, ``$\sum_i(\Delta X_i)^2\to A_t$ in probability''
means convergence of the net of
Proposition~\ref{rem:net-vs-seq}, with $X=L^0(\Omega)$ metrized by the Ky
Fan distance $\rho(U,V):=\E[|U-V|\wedge1]$ (which makes $L^0(\Omega)$ a
genuine, metrizable, hence Hausdorff, topological space) and $x=A_t$.
Kunita's Definition~2.1 in \cite{Kunita82}, by contrast, is genuinely
sequential: it requires $\langle X\rangle_t^{\Delta_m}\to A_t$ in
probability along every sequence $\{\Delta_m\}$ with $|\Delta_m|\to0$,
with no simultaneous quantification over all partitions of a given mesh.
Proposition~\ref{rem:net-vs-seq} shows the two formulations agree; Kunita
does not state this equivalence, working exclusively with sequences
throughout \cite{Kunita82}. It is needed here because the Lemma's
hypothesis, and its use below to identify $A_t$ with $\langle
M\rangle_t$ via \cite{Kunita82}, are stated in the net form, while
Kunita's own results are stated sequentially.

Almost-sure convergence does not fit into this framework for a more
specific reason than mere awkwardness: for $(\Omega,\mathcal F,P)$
non-atomic (in particular, for the Wiener space underlying
Section~\ref{sec:realWiener}), there is no topology $\tau$ on
$L^0(\Omega)$ such that, for every sequence $(X_n)$,
$X_n\xrightarrow{\tau}X$ if and only if $X_n\to X$ almost surely
\cite{Ordman66}. Consequently Proposition~\ref{rem:net-vs-seq} --- which
requires a topological space $X$ with a notion of open neighbourhood
$G\ni x$ to even state --- has no such $X$ to act on when the target mode
of convergence is almost-sure convergence, independently of any metric
consideration. This is consistent with, and sharper than, the
obstruction encountered in the discussion inside the proof of
Lemma~\ref{lem:l2} above, where an almost-sure statement is shown to
require either a fixed sequence of partitions with
$\sum_n\delta(\pi_n)<\infty$ and Borel--Cantelli, or passage to a
subsequence: it is exactly why Kunita's own proof of Theorem~2.2 in
\cite{Kunita82} must pass to a subsequence to upgrade $L^2$ convergence to
almost-sure convergence, rather than asserting it for the net directly.

This obstruction is general -- it holds for arbitrary sequences of
partitions, with no further restriction -- but it is not the end of the
story for the almost-sure quadratic variation of a specific path such
as Brownian motion. Cont and Das~\cite{ContDas23} study exactly the
dependence, pathwise, of $[x]_\pi(t)$ on the choice of partition
sequence $\pi=(\pi^n)$, and isolate a property they call
\emph{quadratic roughness}, along with a restriction to
\emph{balanced} partition sequences (bounded ratio of consecutive
mesh sizes, as for the dyadic sequence). For H\"older paths with this
property -- and they show Brownian paths have it, almost surely --
the almost-sure limit $[x]_\pi(t)$ is the \emph{same} for every
balanced sequence $\pi$, without needing to pass to a subsequence.
This does not contradict Ordman's obstruction, which concerns
arbitrary sequences with no balance restriction; it identifies,
instead, a natural and large class of partition sequences on which
the almost-sure quadratic variation of a typical Brownian path is,
after all, sequence-independent -- precisely the positive counterpart,
for this specific path and this restricted class of sequences, to the
general impossibility recorded above.
\end{remark}

The main example of such a process is a continuous semimartingale
$X_t=Y_t+M_t$, with $Y_t$ a continuous bounded-variation process, $M_t$ a
continuous martingale, and $A_t$ the quadratic variation of $M_t$.
Repeating verbatim the construction above with $\alpha_i=X_{t_i}$,
\[
I_k(\pi)=\ds\sum_{i=0}^{n-1}X_{t_i}^{m-k}X_{t_{i+1}}^k(X_{t_{i+1}}-X_{t_i}),
\]
\[
\rho_{k,j}(\pi)=\sum_{i=0}^{n-1}X_{t_i}^{m-j-1}X_{t_{i+1}}^j(X_{t_{i+1}}-X_{t_i})^2,
\]
and, invoking Lemma~\ref{lem:l2} in place of Lemma~\ref{lem:l1}, exactly
the same computation gives
\[
\rho_k(\pi)\longrightarrow k\int_a^bX_t^{m-1}\,dA_t,
\]
\[
I_0(\pi)\longrightarrow\frac{X_b^{m+1}-X_a^{m+1}}{m+1}-\frac m2\int_a^bX_t^{m-1}\,dA_t,
\]
in probability. Thus
\[
X_b^{m+1}-X_a^{m+1}=(m+1)\int_a^bX(t)^m\,dX_t+\frac{m(m+1)}2\int_a^bX_t^{m-1}\,dA_t,
\]
and, for any polynomial $P$,
\begin{align*}
P(X_b)-P(X_a)&=\int_a^bP'(X(t))\,dX(t)+\frac12\int_a^bP''(X_t)\,dA_t\\
&=\int_a^bP'(X(t))\,dY(t)+\int_a^bP'(X(t))\,dM(t)\\
&\qquad+\frac12\int_a^bP''(X_t)\,dA_t.
\end{align*}

\subsection{The complex (planar) Wiener process}

Let $X_t,Y_t$ be two independent real Wiener processes, and let
$Z_t=X_t+iY_t$ be the corresponding \emph{complex}, or planar, Brownian
motion. We repeat the construction of Section~\ref{sec:realWiener} above,
but now exploiting multiplication in the field $\C$ instead of $\R$ —
exactly as, for the deterministic case, going from the real line to the
complex plane in Section~2 extended the real telescoping identity to
complex multiplication.

Given a partition $\pi$, $a=t_0<t_1<\cdots<t_n=b$, and a fixed integer
$m$, let
\[
I_k(\pi)=\sum_{i=0}^{n-1} Z_{t_i}^{m-k}Z_{t_{i+1}}^k\,(Z_{t_{i+1}}-Z_{t_i}),
\qquad k=0,\ldots,m,
\]
computed with complex multiplication. Exactly as above,
\[
\sum_{k=0}^m I_k(\pi)=Z_b^{m+1}-Z_a^{m+1}
\]
identically, this being a partition-independent algebraic identity;
writing, as before, $I_k(\pi)=I_0(\pi)+\rho_k(\pi)$,
\[
\rho_k(\pi)=\sum_{i=0}^{n-1} Z_{t_i}^{m-k}\big(Z_{t_{i+1}}^{k-1}+
Z_{t_{i+1}}^{k-2}Z_{t_i}+\cdots+Z_{t_i}^{k-1}\big)(Z_{t_{i+1}}-Z_{t_i})^2.
\]

\begin{lemma}\label{lem:l3}
For every integer $k\ge0$,
\[
\sum_{i=0}^{n-1} Z_{t_i}^k\,(Z_{t_{i+1}}-Z_{t_i})^2\ \to\ 0
\]
in $L^2(\Omega;\C)$ (i.e.\ real and imaginary parts both tend to $0$ in
$L^2(\Omega)$) as $\delta(\pi)\to0$.
\end{lemma}

\begin{proof}
Write $\Delta X_i=X_{t_{i+1}}-X_{t_i}=\sqrt{t_{i+1}-t_i}\,\xi_i^{(1)}$ and
$\Delta Y_i=Y_{t_{i+1}}-Y_{t_i}=\sqrt{t_{i+1}-t_i}\,\xi_i^{(2)}$, where
$\xi_i^{(1)},\xi_i^{(2)}$, $i=0,\ldots,n-1$, are i.i.d.\ $N(0,1)$, mutually
independent, and each independent of $Z_{t_i}$. Then
\begin{align*}
(\Delta Z_i)^2&=(\Delta X_i)^2-(\Delta Y_i)^2+2i\,\Delta X_i\Delta Y_i\\
&=(t_{i+1}-t_i)\Big[(\xi_i^{(1)})^2-(\xi_i^{(2)})^2+2i\,\xi_i^{(1)}\xi_i^{(2)}\Big].
\end{align*}
Writing $(\xi_i^{(1)})^2=1+\big[(\xi_i^{(1)})^2-1\big]$ and
$(\xi_i^{(2)})^2=1+\big[(\xi_i^{(2)})^2-1\big]$, the two constant terms
$+1,-1$ cancel \emph{exactly} inside $(\Delta Z_i)^2$ — this is the step
with no real-case analogue, since in $\R$ there is only one $\xi_i$, whose
square has nothing to cancel against, and it is exactly this cancellation
that removes the quadratic-variation term from the complex theory. What
remains is
\begin{align*}
\sum_{i=0}^{n-1} Z_{t_i}^k(\Delta Z_i)^2
&=\sum_{i=0}^{n-1} Z_{t_i}^k\big[(\xi_i^{(1)})^2-1\big](t_{i+1}-t_i)\\
&\qquad-\sum_{i=0}^{n-1} Z_{t_i}^k\big[(\xi_i^{(2)})^2-1\big](t_{i+1}-t_i)\\
&\qquad+2i\sum_{i=0}^{n-1} Z_{t_i}^k\,\xi_i^{(1)}\xi_i^{(2)}(t_{i+1}-t_i).
\end{align*}
Each of the three sums on the right is, term by term, centred and
independent of $Z_{t_i}$ (using $\E[(\xi_i^{(1)})^2-1]=\E[(\xi_i^{(2)})^2-1]
=\E[\xi_i^{(1)}\xi_i^{(2)}]=0$), so exactly the estimate used in the proof
of Lemma~\ref{lem:l1} for $\ell=2$ applies to each of them separately
(with $\E[Z_t^k]$-type constants replaced by the finite moment bound
$\sup_{a\le t\le b}\E|Z_t|^{2k}<\infty$, since $Z_t$ is Gaussian): each
sum has second moment bounded by a constant times
$\sum_i(t_{i+1}-t_i)^2\le(b-a)\delta(\pi)\to0$. Hence the whole expression
tends to $0$ in $L^2(\Omega;\C)$.
\end{proof}

\begin{thm}\label{thm:complexIto}
For every integer $m\ge0$,
\[
\int_a^b Z_t^m\,dZ_t:=\lim_{\delta(\pi)\to0}I_0(\pi)=
\frac{Z_b^{m+1}-Z_a^{m+1}}{m+1}\qquad\text{in }L^2(\Omega;\C),
\]
with \emph{no} correction term; and for every polynomial
$P(z)=\sum_j c_jz^j$,
\[
P(Z_b)-P(Z_a)=\int_a^b P'(Z_t)\,dZ_t.
\]
\end{thm}

\begin{proof}
Each $\rho_k(\pi)$ is a finite $\C$-linear combination, with
partition-independent coefficients, of sums of the type appearing in
Lemma~\ref{lem:l3} (for the powers $0,\ldots,k-1$ of $Z_{t_i}$), so
$\rho_k(\pi)\to0$ in $L^2(\Omega;\C)$ for every $k=1,\ldots,m$. From
$\sum_kI_k(\pi)=Z_b^{m+1}-Z_a^{m+1}$ and $I_k(\pi)=I_0(\pi)+\rho_k(\pi)$,
\[
(m+1)I_0(\pi)=Z_b^{m+1}-Z_a^{m+1}-\sum_{k=1}^m\rho_k(\pi)\ \longrightarrow\
Z_b^{m+1}-Z_a^{m+1}
\]
in $L^2(\Omega;\C)$, which gives the first identity. Linearity extends it
to any polynomial $P$.
\end{proof}

\begin{remark}
Theorem~\ref{thm:complexIto} is the elementary, Riemann-sum counterpart of
the classical fact that a holomorphic function of planar Brownian motion
is a (local) martingale. Writing $P(z)=u(x,y)+iv(x,y)$ for the real and
imaginary parts of a holomorphic polynomial, the Cauchy--Riemann equations
make $u,v$ harmonic ($\Delta u=\Delta v=0$), so the two-dimensional
It\^o formula applied to $u(X_t,Y_t)$ would read
\[
u(X_b,Y_b)-u(X_a,Y_a)=\int_a^b\nabla u(X_t,Y_t)\cdot(dX_t,dY_t)
+\frac12\int_a^b\Delta u(X_t,Y_t)\,dt,
\]
and the correction term $\frac12\Delta u\,dt$ vanishes identically because
$u$ is harmonic — likewise for $v$. The computation in the proof of
Lemma~\ref{lem:l3} reproves this cancellation directly at the Riemann-sum
level, without invoking the Laplacian or harmonicity as a black box: it is
simply the observation that $\E[(\xi_i^{(1)})^2]=\E[(\xi_i^{(2)})^2]=1$
cancel against each other inside the complex square $(\Delta Z_i)^2$,
whereas in the real one-dimensional case $\E[\xi_i^2]=1$ has nothing to
cancel against and survives as the quadratic-variation term
$\int_a^bW_t^{m-1}\,dt$ of the real It\^o formula
(Section~\ref{sec:realWiener}). In this
sense the vanishing of the correction term for holomorphic $P$ is, at
bottom, the same elementary cancellation
$\lambda^{m+1}-\mu^{m+1}=(\lambda-\mu)(\cdots)$ used throughout this note,
specialised to $\C$: the vanishing is reflected, at the Riemann-sum
level, in the cancellation occurring inside the complex square
$(\Delta Z_i)^2$, rather than in any analytic property of $P$ itself.
\end{remark}

\subsection{Generalization: conformal martingales}

The cancellation established in Lemma~\ref{lem:l3} is not special to the
independent pair $(X,Y)$: it holds under a purely quadratic-variation
condition on the pair $(A,B)$ making up $M=A+iB$, stated below, and the
elementary, Riemann-sum proof above extends to it with no new ideas —
only a bilinear version of Lemma~\ref{lem:l2}. We state the hypotheses
first, in their weakest (abstract, non-probabilistic) form, and only
afterwards attach the classical name to the resulting object, to avoid
any suggestion that the martingale property itself is being used.

Let $A_t,B_t$ be continuous real processes such that, in probability as
$\delta(\pi)\to0$ and for every $t\in[a,b]$,
\[
\sum_i(\Delta A_i)^2\to\langle A\rangle_t,\qquad
\sum_i(\Delta B_i)^2\to\langle B\rangle_t,
\]
\[
\sum_i\big(\Delta(A+B)_i\big)^2\to\langle A+B\rangle_t,
\]
with $\langle A\rangle,\langle B\rangle,\langle A+B\rangle$ continuous and
non-decreasing (the hypotheses of Lemma~\ref{lem:l2}, applied
\emph{separately} to $A$, to $B$, and to $A+B$ --- three independent
hypotheses, not four: we do \emph{not} separately assume existence of a
limit for the cross sum $\sum_i\Delta A_i\Delta B_i$, since it is
constructed below by polarization from these three, and need not be
posited in advance). Set $M_t:=A_t+iB_t$, and \emph{define} the
quadratic covariation by polarization,
\[
\langle A,B\rangle_t:=\tfrac12\big(\langle A+B\rangle_t-\langle
A\rangle_t-\langle B\rangle_t\big).
\]

\begin{lemma}\label{lem:bilinear}
For any continuous adapted process $\Phi_t$,
\[
\sum_i\Phi(t_i)\,\Delta A_i\Delta B_i\ \xrightarrow[\delta(\pi)\to0]{P}\
\int_a^b\Phi_t\,d\langle A,B\rangle_t.
\]
\end{lemma}
\begin{proof}
By polarization, with $U:=A+B$,
\[
2\,\Delta A_i\Delta B_i=(\Delta U_i)^2-(\Delta A_i)^2-(\Delta B_i)^2.
\]
By hypothesis $U=A+B$ satisfies the hypothesis of Lemma~\ref{lem:l2}
directly (with limit $\langle U\rangle=\langle A+B\rangle$), as do $A,B$
separately. That lemma's proof (Steps~1--2) uses only the continuity of
the weight process, never its specific form as a power $X_t^k$: the
identical argument, with $\Phi$ in place of $X_t^k$, gives
$\sum_i\Phi(t_i)(\Delta U_i)^2\to\int_a^b\Phi_t\,d\langle U\rangle_t$ and
likewise for $A,B$, all in probability. Hence
\[
\sum_i\Phi(t_i)\,\Delta A_i\Delta B_i
=\tfrac12\Big[\sum_i\Phi(t_i)(\Delta U_i)^2-\sum_i\Phi(t_i)(\Delta A_i)^2
\]
\[
-\sum_i\Phi(t_i)(\Delta B_i)^2\Big]
\]
\[
\longrightarrow\tfrac12\int_a^b\Phi_t\,d\big(\langle U\rangle_t-\langle
A\rangle_t-\langle B\rangle_t\big)=\int_a^b\Phi_t\,d\langle A,B\rangle_t,
\]
using the definition of $\langle A,B\rangle$ in the last step.
\end{proof}

\begin{thm}\label{thm:conformalIto}
If $\langle M,M\rangle_t:=\langle A\rangle_t-\langle B\rangle_t+2i\langle
A,B\rangle_t\equiv0$ for all $t$ (equivalently, $\langle A\rangle_t=
\langle B\rangle_t$ and $\langle A,B\rangle_t\equiv0$), then for every
integer $m\ge0$,
\[
\int_a^b M_t^m\,dM_t:=\lim_{\delta(\pi)\to0}I_0(\pi)=
\frac{M_b^{m+1}-M_a^{m+1}}{m+1}\qquad\text{in probability},
\]
with \emph{no} correction term; and for every polynomial $P$,
\[
P(M_b)-P(M_a)=\int_a^b P'(M_t)\,dM_t.
\]
\end{thm}

\begin{proof}
As in the proof of Theorem~\ref{thm:complexIto}, each $\rho_k(\pi)$ is a
finite $\C$-linear combination of sums $\sum_iM_{t_i}^j(\Delta M_i)^2$,
$j=0,\ldots,k-1$. Writing $(\Delta M_i)^2=(\Delta A_i)^2-(\Delta
B_i)^2+2i\,\Delta A_i\Delta B_i$ and applying Lemma~\ref{lem:l2} (for the
first two sums, with weight $\Phi=M_t^j$ — again valid for any continuous
weight, not only powers of the process whose increments are squared) and
Lemma~\ref{lem:bilinear} (for the cross term) gives
\begin{align*}
\sum_iM_{t_i}^j(\Delta M_i)^2\ &\xrightarrow[\delta(\pi)\to0]{P}\
\int_a^bM_t^j\,d\langle A\rangle_t-\int_a^bM_t^j\,d\langle B\rangle_t\\
&\qquad+2i\int_a^bM_t^j\,d\langle A,B\rangle_t
=\int_a^bM_t^j\,d\langle M,M\rangle_t=0,
\end{align*}
using the hypothesis $\langle M,M\rangle\equiv0$ in the last equality. Hence
$\rho_k(\pi)\to0$ in probability for $k=1,\ldots,m$, and the conclusion
follows exactly as in the proof of Theorem~\ref{thm:complexIto}, with
convergence in probability in place of $L^2(\Omega;\C)$.
\end{proof}

\begin{definition}[\cite{GS}]\label{def:conformal}
When $A,B$ are honest continuous local martingales (with respect to some
filtration) satisfying $\langle M,M\rangle\equiv0$ as in
Theorem~\ref{thm:conformalIto}, $M=A+iB$ is called a \emph{conformal
(local) martingale}.
\end{definition}

\begin{remark}[On the terminology]
Definition~\ref{def:conformal} attaches the classical name to the
classical situation, but Theorem~\ref{thm:conformalIto} itself is proved
under strictly weaker hypotheses: only the quadratic-variation
convergence properties of Lemma~\ref{lem:l2} (in probability, along the
net $\delta(\pi)\to0$) for $A$, $B$, and $A+B$, together with
$\langle M,M\rangle\equiv0$ -- never the martingale property itself
(adaptedness and the conditional-expectation identity) for $A$ or $B$.
The theorem therefore applies to any pair $(A,B)$ with the stated
quadratic-variation behaviour and vanishing bracket, whether or not
$A,B$ are genuine local martingales; Definition~\ref{def:conformal} is
only invoked when they happen to be, in which case
Theorem~\ref{thm:conformalIto} recovers the classical statement about
conformal local martingales of \cite{GS}.
\end{remark}

\begin{remark}
Theorem~\ref{thm:complexIto} is the special case $A=X$, $B=Y$ independent
Wiener processes: conformality is then immediate from independence
($\langle X\rangle_t=\langle Y\rangle_t=t$, $\langle X,Y\rangle_t\equiv0$),
and the explicit Gaussian moment computation in the proof of
Lemma~\ref{lem:l3} upgrades the mode of convergence from probability to
$L^2(\Omega;\C)$ — an upgrade the abstract hypotheses of this subsection
do not, by themselves, provide.
\end{remark}

\begin{remark}[The non-conformal case]
Without the conformal hypothesis, the same computation gives instead
\[
\rho_k(\pi)\ \xrightarrow[\delta(\pi)\to0]{P}\ k\int_a^bM_t^{m-1}\,d\langle
M,M\rangle_t,
\]
and, exactly as in Section~\ref{sec:realWiener},
\[
M_b^{m+1}-M_a^{m+1}=(m+1)\int_a^bM_t^m\,dM_t+\frac{m(m+1)}2\int_a^bM_t^{m-1}\,d\langle
M,M\rangle_t,
\]
so that, for any polynomial $P$,
\[
P(M_b)-P(M_a)=\int_a^bP'(M_t)\,dM_t+\frac12\int_a^bP''(M_t)\,d\langle
M,M\rangle_t,
\]
the classical complex It\^o formula, driven by the \emph{complex}
bracket $\langle M,M\rangle$ (bilinear, not the positive-definite
$\langle M,\overline M\rangle=\langle A\rangle+\langle B\rangle$).
Theorem~\ref{thm:conformalIto} is the case $\langle M,M\rangle\equiv0$;
setting $B\equiv0$ instead recovers the real semimartingale formula of
Section~\ref{sec:realWiener}, since then $M=A$ and $\langle
M,M\rangle=\langle A\rangle$. The two cases developed earlier in this
note — the real semimartingale correction term and its complete
cancellation for planar Brownian motion — are thus the two extremes,
$B\equiv0$ and $M$ conformal, of a single formula.
\end{remark}

\begin{remark}[Concluding remark: the quadratic remainder, one more time]
This is the point flagged in the abstract as the paper's most
distinctive single observation, and it is worth restating plainly now
that all the pieces are in place. Every case in this note rests on the
same algebraic telescoping identity and the same quadratic remainder
term $\sum_i(\text{weight})_i(\Delta_i)^2$; what varies is only that
term's fate. On $\R$ it vanishes identically (Part~I). On $\C$, for
H\"older paths, and along the strange nets of Section~2, it vanishes in
the limit precisely because the standing hypothesis $Q(\pi)\to0$ forces
it to (Part~II). For a real continuous semimartingale it neither
vanishes nor is negligible: it survives as the quadratic variation and
becomes, verbatim, the second-derivative correction of It\^o's formula.
And for a planar Brownian motion, or more generally a conformal
martingale, it is present in exactly the same algebraic form as in the
real case -- yet cancels completely, for the purely algebraic reason
that $\E[(\xi^{(1)})^2]=\E[(\xi^{(2)})^2]$ inside the complex square
$(\Delta Z_i)^2$, with nothing analogous for a single real Gaussian
increment to cancel against. The classical fact that a holomorphic
function of planar Brownian motion is again a local martingale, and the
classical It\^o correction term for a real diffusion, are ordinarily
presented as two separate facts about two different processes; the
present construction exhibits them as the same quadratic remainder,
evaluated in $\R$ and in $\C$ respectively.
\end{remark}

\appendix
\section{Uniform convergence and termwise integration}

We record here, for completeness, the classical fact used implicitly in
Section~\ref{sec:cont-riemann} to pass from polynomials to continuous functions, together
with an illustrative example.

\begin{thm}
Let $\{f_n\}$ be functions $f_n\colon[a,b]\to\R$, each Riemann integrable,
with $f_n\to f$ uniformly on $[a,b]$. Then $f$ is Riemann integrable and
$\int_a^bf_n(x)\,dx\to\int_a^bf(x)\,dx$.
\end{thm}
\begin{proof}
Write $I_n=\int_a^bf_n(x)\,dx$ and, for a tagged partition $\sigma$ with
tags $\xi_i$, $S(g,\sigma)=\sum_ig(\xi_i)(x_{i+1}-x_i)$. Fix $\eps>0$ and
$n_\eps$ such that $|f_n(x)-f(x)|<\eps$ for $n\ge n_\eps$ and all
$x\in[a,b]$. Then, for every $\sigma$ and $n\ge n_\eps$,
\begin{equation}\label{eq:app1}
|S(f_n,\sigma)-S(f,\sigma)|\le\sum_i|f_n(\xi_i)-f(\xi_i)|(x_{i+1}-x_i)<\eps(b-a).
\end{equation}
For two partitions $\sigma_1,\sigma_2$ and $n\ge n_\eps$,
\[
|S(f,\sigma_1)-S(f,\sigma_2)|
\le|S(f_n,\sigma_1)-S(f_n,\sigma_2)|+2\eps(b-a),
\]
and since $f_n$ is Riemann integrable, $|S(f_n,\sigma_1)-S(f_n,\sigma_2)|\to0$
as $|\sigma_1|,|\sigma_2|\to0$. Hence $\{S(f,\sigma)\}$ is Cauchy as
$|\sigma|\to0$ and converges to some $I=:\int_a^bf(x)\,dx$; letting
$|\sigma|\to0$ in \eqref{eq:app1} gives $|I_n-I|\le\eps(b-a)$ for
$n\ge n_\eps$, so $I_n\to I$.
\end{proof}

\begin{remark}[An example]
Let $f_n(x)=x-\dfrac{x^3}{3!}+\cdots+(-1)^{n-1}\dfrac{x^{2n-1}}{(2n-1)!}$.
Then $f_n\to\sin x$ uniformly on bounded sets, so
$\int_a^bf_n(x)\,dx\to\int_a^b\sin(x)\,dx$. On the other hand, by
\eqref{real:eq0} applied term by term,
\begin{align*}
\int_a^bf_n(x)\,dx
&=\Big[\frac{b^2}{2!}-\frac{b^4}{4!}+\cdots\Big]-\Big[\frac{a^2}{2!}-\frac{a^4}{4!}+\cdots\Big]\\
&=\Big[1-\cos b+o(1)\Big]-\Big[1-\cos a+o(1)\Big],
\end{align*}
so that, in the limit, $\int_a^b\sin(x)\,dx=\cos a-\cos b$, recovering the
elementary formula directly from the argument of
Section~\ref{sec:real-key}.
\end{remark}

\end{document}